\documentclass[10pt, reqno]{amsart}

\usepackage{amsmath,amssymb,amsthm,amsfonts,verbatim}
\usepackage{microtype}
\usepackage[all,2cell]{xy}
\usepackage{mathtools}
\usepackage{graphicx}
\usepackage{hyperref}
\usepackage{mathrsfs}

\CompileMatrices

\usepackage[top=1.3in,bottom=1.9in,left=1.4in,right=1.4in]{geometry}

\theoremstyle{plain}
\newtheorem{theorem}{Theorem}[section]

\newtheorem{proposition}[theorem]{Proposition}
\newtheorem{lemma}[theorem]{Lemma}

\theoremstyle{definition}

\newtheorem{remark}[theorem]{Remark}

\newcommand{\nc}{\newcommand}
\nc{\dmo}{\DeclareMathOperator}

\nc{\Q}{\mathbb{Q}}
\nc{\F}{\mathbb{F}}
\nc{\R}{\mathbb{R}}
\nc{\Z}{\mathbb{Z}}
\nc{\C}{\mathbb{C}}
\nc{\Ell}{\mathcal{L}}
\nc{\M}{\mathcal{M}}
\nc{\K}{\mathcal{K}}
\nc{\disk}{\mathbb{D}}
\nc{\hyp}{\mathbb{H}}

\nc{\CP}{\mathbb{CP}}
\nc{\cS}{\mathcal{S}}
\dmo{\Mod}{Mod}
\dmo{\Diff}{Diff}
\dmo{\Homeo}{Homeo}
\dmo{\dist}{dist}
\dmo\BDiff{BDiff}
\dmo\SO{SO}
\dmo\Hom{Hom}
\dmo\SL{SL}
\dmo\Sp{Sp}
\dmo\rank{rank}
\dmo\sig{sig}
\dmo\Out{Out}
\dmo\Aut{Aut}
\dmo\Inn{Inn}
\dmo\GL{GL}
\dmo\PSL{PSL}
\dmo\tr{tr}
\dmo\BHomeo{BHomeo}
\dmo\EHomeo{EHomeo}
\dmo\EDiff{EDiff}
\nc\Sig{\Sigma}
\dmo\Teich{Teich}
\dmo\Fix{Fix}
\nc{\pair}[1]{\langle #1 \rangle}
\nc{\abs}[1]{\left| #1 \right|}
\nc{\action}{\circlearrowright}
\nc{\norm}[1]{\left | \left | #1 \right | \right |}
\nc{\abcd}[4]{\left(\begin{array}{cc} #1 & #2 \\ #3 & #4 \end{array}\right)}
\nc{\into}\hookrightarrow
\dmo{\Isom}{Isom}
\nc{\normal}{\vartriangleleft}
\dmo{\Vol}{Vol}
\dmo{\im}{Im}
\dmo{\Push}{Push}
\dmo{\Conf}{Conf}
\dmo{\PConf}{PConf}
\dmo{\id}{id}

\renewcommand{\epsilon}{\varepsilon}
\nc{\coloneq}{\mathrel{\mathop:}\mkern-1.2mu=}
\nc{\margin}[1]{\marginpar{\scriptsize #1}}
\nc{\para}[1]{\medskip\noindent\textbf{#1.}}

\newenvironment{packed_enum}{
\begin{enumerate}
  \setlength{\itemsep}{0pt}
  \setlength{\parskip}{0pt}
  \setlength{\parsep}{0pt}
}{\end{enumerate}}

\title[Cup products, Johnson homomorphism, and surface bundles over surfaces]{Cup products, the Johnson homomorphism, and surface bundles over surfaces with multiple fiberings}

\author{Nick Salter}
\email{nks@math.uchicago.edu}
\date{\today}

\address{Department of Mathematics\\ University of Chicago\\ 5734 S. University Ave., Chicago, IL 60637}

\begin{document}
\maketitle

\begin{abstract}
Let $\Sigma_g \to E \to \Sigma_h$ be a surface bundle over a surface with monodromy representation $\rho: \pi_1 \Sigma_h \to \Mod(\Sigma_g)$ contained in the Torelli group $\mathcal{ I}_g$. In this paper we express the cup product structure in $H^*(E, \Z)$ in terms of the Johnson homomorphism $\tau: \mathcal{ I}_g \to \wedge^3 (H_1 (\Sigma _g, \Z))$. This is applied to the question of obtaining an upper bound on the maximal $n$ such that $p_1: E \to \Sigma_{h_1}, \dots, p_n: E \to \Sigma_{h_n}$ are fibering maps realizing $E$ as the total space of a surface bundle over a surface in $n$ distinct ways. We prove that any nontrivial surface bundle over a surface with monodromy contained in the Johnson kernel $\K_g$ fibers in a unique way.

\end{abstract}

\section{Introduction}  The theory of the Thurston norm gives a detailed picture of the set of possible ways that a compact, oriented $3$-manifold $M$ can fiber as a surface bundle. If $b_1(M) > 1$, then $M$ admits infinitely many such fibrations $\Sigma_g \to M \to S^1$; finitely many for each $g \ge 2$. The purpose of the present paper is to take up a similar sort of inquiry for $4$-manifolds $\Sigma_g \to E \to \Sigma_h$ fibering as a surface bundle over a surface of genus $g \ge 2$. 

When $h = 1$ (i.e. the base surface is a torus), a similar story as in the $3$-manifold setting unfolds; if $M^3$ is a $3$-manifold admitting infinitely many fiberings $p: M \to S^1$, then $p\times \id: M^3\times S^1 \to S^1 \times S^1$ admits infinitely many fiberings as well. However, in stark contrast with the $3$-dimensional setting and with the case of surface bundles over the torus, F.E.A. Johnson showed in \cite{FEA2} that if $\Sigma_g \to E \to \Sigma_h$ is a surface bundle over a surface with $g, h \ge 2$, then there are only finitely many distinct fibrations $p_i: E \to \Sigma_{h_i}$ realizing $E$ as the total space of a surface bundle over a surface (see Proposition \ref{proposition:fiberings} for a precise definition of what is meant by ``distinct''). The book \cite{hillman} contains a treatment of results of this type, as does the preprint \cite{rivinfiber}, in which the case of surface bundles over surfaces is situated in the larger context of ``fibering rigidity'' for a wide class of manifolds. 

A particularly simple example of a surface bundle over a surface admitting two fiberings is that of a trivial bundle, i.e. a product of surfaces $\Sigma_g \times \Sigma_h$. At the time of Johnson's result, there was essentially one known method for producing nontrivial surface bundles over surfaces with multiple fiberings, due independently to Atiyah and Kodaira (see \cite{atiyah}, \cite{kodaira}, as well as the summary in \cite{moritabook}). Their construction is built by taking a certain cyclic branched covering $p: E \to \Sigma_g \times \Sigma_h$ of a product of surfaces. The two fibering maps are inherited from the projections of $\Sigma_g \times \Sigma_h$ onto either factor. While Johnson's argument produces a bound on the number of possible fiberings of a surface bundle $E$ that is super-exponential in the Euler characteristic $\chi(E)$, until recently all known examples of surface bundles over surfaces had at most two fiberings, leaving a large gap between upper and lower bounds on the number of possible fiberings.

In \cite{salterconstruction}, the author gave a new method for constructing surface bundles over surfaces with multiple fiberings, including the first examples of bundles admitting an arbitrarily large number of fiberings. In fact, the methods of \cite{salterconstruction} are capable of producing families $E_n$ of surface bundles admitting exponentially many fiberings as a function of $\chi(E_n)$. The results of this paper can be seen as a complement to the work of \cite{salterconstruction}, in that our concern here is in addressing the question of when surface bundles over surfaces admit {\it unique} fiberings.\\

 A central theme in the study of surface bundles is the ``monodromy - topology dictionary''. For any reasonable base space $M$, there is a well-known correspondence (see, e.g. \cite{FM})
\begin{equation}\label{eqn:correspondence}
\left\{ \begin{array}{c} \text{Bundle-isomorphism classes of}\\ \text{oriented }\Sigma_g\text{-bundles over } M\end{array} \right\} 
\longleftrightarrow \left\{ \begin{array}{c}\text{Conjugacy classes of representations}\\ \pi_1(M) \to \Mod(\Sigma_g) \end{array}\right\}.
\end{equation}

This raises the question of translating between topological and geometric properties of surface bundles on the one hand, and on the other, algebraic or geometric properties of the monodromy representation. Certain entries in this dictionary are well-established, for instance Thurston's landmark result that a fibered 3-manifold $\Sigma_g \to M_\phi \to S^1$ admits a complete hyperbolic metric if and only if the monodromy is a so-called ``pseudo-Anosov'' element of $\Mod(\Sigma_g)$. In this paper we add to the dictionary by relating the cohomology ring of a surface bundle over a surface to its monodromy representation, and apply these results to give various obstructions for the surface bundle to admit more than one fibering.

From the perspective of the monodromy representation, the phenomenon of multiple fibering remains mysterious. The central result of this paper shows that there is a strong interaction between the existence of multiple fiberings and the theory of the Torelli group $\mathcal I_g$. Recall that the {\it Torelli group} is the kernel of the symplectic representation $\Psi: \Mod(\Sigma_g) \to \Sp_{2g}(\Z)$ and that the {\em Johnson kernel} $\mathcal K_g$ is defined as the group generated by Dehn twists $T_\gamma$ with $\gamma$ a {\em separating} curve. 

\begin{theorem}\label{theorem:double}
Let $\pi: E \to B$ be a surface bundle over a surface with monodromy in the Johnson kernel $\mathcal{K}_g$. If $E$ admits two distinct fiberings then $E$ is diffeomorphic to $B \times B'$, the product of the base spaces. In other words, any nontrivial surface bundle over a surface with monodromy in $\K_g$ admits a unique fibering.
\end{theorem}

The surface bundles over surfaces of \cite{salterconstruction} can be constructed so as to have monodromy contained in $\mathcal I _g$. It follows that the hypothesis in Theorem \ref{theorem:double} that the monodromy be contained in $\mathcal K_g$ is effectively {\it sharp} with respect to the {\it Johnson filtration} (see Chapter 6 of \cite{FM} for the definition of the Johnson filtration).\\

Theorem \ref{theorem:double} is proved by first relating the monodromy representation of a surface bundle over a surface $E^4 \to B^2$ to the cohomology ring $H^*(E)$; it is then shown that obstructions to possessing alternative fiberings can be extracted from $H^*(E)$. In a similar spirit we also have the following general criterion which we believe to be of independent interest, for a surface bundle over a surface to possess a unique fibering.

Recall that if $\rho: G \to \GL(V)$ is a representation, then the {\em invariant} space $V^\rho$ is defined via
\[
V^\rho = \{v \in V: \rho(g)(v) = v \mbox{ for all } g \in G\}.
\]
The space of {\em co-invariants} $V_\rho$ of the representation is defined as
\[
V_\rho = V / W, \qquad \mbox{ where }\qquad W = \{v - \rho(g)(v) | v \in V, g \in G\}.
\]

\noindent \textbf{Theorem \ref{theorem:single}.} \textit{Let $p: E \to B$ be a surface bundle over a surface $B$ of genus $g \ge 2$ with monodromy representation $\rho: \pi_1 B \to \Mod(\Sigma_g)$. Suppose the space of invariant cohomology $(H^1(F,\Q))^\rho$ (equivalently, the coinvariant homology of the fiber $(H_1(F,\Q))_\rho$) vanishes. Then $E$ admits a unique fibering.
}\\

Recall that a surface bundle over $S^1$, viewed as the mapping torus $M$ of some diffeomorphism $\phi$, admits a unique fibering if and only if $b_1(M) = 1$. This is the case exactly when $(H_1(M, \Q))_\phi = 0$, so Theorem \ref{theorem:single} is the counterpart to this fact in dimension $4$. Moreover, a random element $\phi \in \Mod(\Sigma_g)$ satisfies $(H_1(M, \Q))_\phi=0$ (see \cite{rivin}). It easily follows that a generic monodromy representation will also have $(H_1(E, \Q))_\rho = 0$: ``most'' surface bundles over surfaces have a single fibering. The proof of Theorem \ref{theorem:single} is special to the case of surface bundles over surfaces and it is not clear if Theorem \ref{theorem:single} is true in greater generality.\\

The paper is organized as follows. In Section \ref{section:equivalence}, we give various characterizations of the notion of equivalence under consideration. In Section \ref{section:single}, we prove Theorem \ref{theorem:single}. Sections \ref{section:biproj} - \ref{section:proof} are devoted to the proof of Theorem \ref{theorem:double}. Section \ref{section:biproj} is devoted to a lemma in differential topology that features in later stages of the proof of Theorem \ref{theorem:double}. The technical heart of the paper is Section \ref{section:johnson}. In it, we first give an overview of the classical description of the Johnson homomorphism $\tau$ in terms of the intersection theory of surfaces in $3$-manifolds that fiber over $S^1$. Using this description of $\tau$ we then carry out a construction of $3$-manifolds embedded in surface bundles over surfaces that realizes the relationship between the Johnson homomorphism and the intersection product in the homology of the surface bundle. We give a complete description of the product structure in (co)homology for a surface bundle over a surface with monodromy in $\mathcal I_g$. These methods of Section \ref{section:johnson} extend to an arbitrary surface bundle over a surface, but we do not state them in this level of generality since we have no need for them here.

Section \ref{section:appendix} is devoted to some technical results concerning multisections of surface bundles, and their connection to splittings on rational cohomology. These results are used in the course of proving Theorem \ref{theorem:double}.

In Section \ref{section:proof} we turn finally to the proof of Theorem \ref{theorem:double}. The result follows from an analysis of the intersection product structure in $H_* E$ for a surface bundle over a surface $\Sigma_g \to E \to \Sigma_h$ with monodromy in $\K_g$. The results of Section \ref{section:johnson} are applied to show that when the monodromy of $\Sigma_g \to E \to \Sigma_h$ is contained in $\K_g$, then $E$, which necessarily has $H^*E \approx H^*\Sigma_g \otimes H^*\Sigma_h$ as an additive group, in fact has $H^*E \approx H^*\Sigma_g \otimes H^*\Sigma_h$ (with $\Z$ coefficients) as a graded ring. This condition is then exploited to prove Theorem \ref{theorem:double}. 

\para{Acknowledgements} The author would like to express his gratitude to Tom Church, Sebastian Hensel, Jonathan Hillman, Andy Putman, and Alden Walker for illuminating discussions at various stages of this work. He is grateful to an anonymous referee for many helpful suggestions. He would also like to extend his warmest thanks to Benson Farb for his extensive comments as well as his invaluable support from start to finish.


\section{Equivalence}\label{section:equivalence}
If $E$ is a smooth $n$-manifold and $p_i: E \to B_i,\ i = 1, \dots, k$ are projection maps for various fiber bundle structures on $E$, we can consider the product of all the projection maps:
\[
p_1 \times \dots \times p_k : E \to B_1 \times \dots \times B_k.
\]
In particular, if $E^4$ is the total space of a surface bundle over a surface with two fiberings, the {\em bi-projection} $p_1 \times p_2: E \to B_2 \times B_2$ is defined. As remarked in the introduction, ultimately we are concerned with fiberwise-diffeomorphism classes of surface bundles. However, it is convenient to consider a more restrictive notion of equivalence which will turn out to have the advantage of being describable purely on the level of the fundamental group.\\
\indent We say that two fiberings $p_1: E \to B_1$, $p_2: E \to B_2$ are {\em $\pi_1$-fiberwise diffeomorphic} if $(1)$ - they are fiberwise diffeomorphic, i.e. there exists a commutative diagram
\[
\xymatrix{
E \ar[r]^\phi \ar[d]_{p_1}		& E \ar[d]^{p_2}\\
B_1 \ar[r]_\alpha			& B_2
}
\]
with $\phi, \alpha$ diffeomorphisms, and $(2)$ - $\phi_*(\pi_1 F_1) = \pi_1 F_1$ (here, as always, $F_i$ denotes a fiber of $p_i$). Certainly if $p_1, p_2$ are $\pi_1$-fiberwise diffeomorphic bundle structures, then they are fiberwise-diffeomorphic in the usual sense. We are interested in this notion because we want to always regard the trivial bundle $\Sigma_g \times \Sigma_h$ as having two distinct fiberings. In the setting of fiberwise-diffeomorphism, the projections onto either factor of $\Sigma_g \times \Sigma_g$ yield equivalent fiberings via the factor-swapping map $\phi(x,y) = (y,x)$, which covers the identity on $\Sigma_g$, but $\phi_*(\pi_1(\Sigma_g \times \{p\})) \ne \pi_1(\Sigma_g \times \{p\})$. The following proposition asserts that $\pi_1$-fiberwise diffeomorphism classes are in correspondence with the fiber subgroups $\pi_1 F \normal \pi_1 E$. Recall that this is the setting in which F.E.A. Johnson proved his finiteness result (see \cite{FEA2}).

\begin{proposition}\label{proposition:fiberings}
Suppose $E$ is the total space of a surface bundle over a surface in two ways: $p_1: E \to B_1$ and $p_2: E \to B_2$. Let $F_1, F_2$ denote fibers of $p_1, p_2$ respectively. Then the following are equivalent:
\begin{enumerate}
\item The fiberings $p_1, p_2$ are $\pi_1$-fiberwise diffeomorphic.
\item The fiber subgroups $\pi_1 F_1, \pi_1 F_2 \le \pi_1 E$ are equal.
\end{enumerate}
If $ \deg (p_1 \times p_2) \ne 0$ then the bundle structures $p_1$ and $p_2$ are distinct. 
\end{proposition}

\begin{proof} 
First suppose that $p_1$ and $p_2$ are equivalent. Appealing to the long exact sequence in homotopy, we see that
\[
\xymatrix{
1 \ar[r]	& \pi_1 F_1 \ar[r] \ar[d]_{\phi_*}	& \pi_1 E \ar[r] \ar[d]_{\phi_*}	& \pi_1 B_1 \ar[r] \ar[d]_{\alpha_*} 	& 1\\
1 \ar[r]	& \pi_1 F_2 \ar[r]			& \pi_1 E \ar[r]				& \pi_1 B_2 \ar[r]			& 1.
}
\]
By assumption $\phi_*(\pi_1 F_1) = \pi_1 F_1$, so that (\ref{proposition:fiberings}.1) $\implies$ (\ref{proposition:fiberings}.2) as claimed.\\
\indent Conversely, suppose that $\pi_1 F_1 = \pi_1 F_2$. Therefore the bundle structures $p_1$ and $p_2$ give rise to the same splitting
\[
1 \to \pi_1 F \to \pi_1 E \to \pi_1 B \to 1
\]
on fundamental group. The monodromy for each bundle can be obtained from this sequence via the map $\pi_1 B \to \Out(\pi_1 F) \approx \Mod(\Sigma_g)$. This shows that the monodromies for the two bundle structures are conjugate, and so via the correspondence (\ref{eqn:correspondence}), there is a bundle-isomorphism $\phi: E \to E$ covering the identity on $B$. 
To see that $\phi_*(\pi_1 F_1) = \pi_1 F_1$, consider the induced map on the long exact sequence in homotopy coming from $\phi$:
\[
\xymatrix{
1 \ar[r]	& \pi_1 F_1 \ar[r] \ar[d]_{\phi_*}	& \pi_1 E \ar[r] \ar[d]_{\phi_*}	& \pi_1 B \ar[r] \ar@{=}[d] 	& 1\\
1 \ar[r]	& \pi_1 F_2 \ar[r]				& \pi_1 E \ar[r]				& \pi_1 B \ar[r]			& 1.
}
\]
This shows $\phi_*(\pi_1 F_1) = \pi_1 F_2$, and $\pi_1 F_1 = \pi_1 F_2$ by assumption.\\

Having established the equivalence of $(\ref{proposition:fiberings}.1)$ and $(\ref{proposition:fiberings}.2)$, it remains to show that if $\deg (p_1 \times p_2) \ne 0$, then $p_1$ and $p_2$ are distinct. We establish the contrapositive. Suppose that $\pi_1 F_1 = \pi_1 F_2$. For $i = 1,2$, we view $\pi_1 B_i$ as the quotient $\pi_1 B_i \approx \pi_1 E / \pi_1 F_i$. If $p_1 \times p_2$ is the bi-projection, then in this notation,
\[
(p_1 \times p_2)_*: \pi_1 E \to \pi_1 B_1 \times \pi_1 B_2
\]
is given by
\[
(p_1 \times p_2)_*(x) = (x\,\pi_1 F_1, x\,\pi_1 F_2) = ([x],[x]),
\]
where $[x] = x \pmod{\pi_1 F_1} = x\pmod{\pi_1 F_2}$. As $\pi_1 F_1 = \pi_1 F_2$, the quotients $\pi_1 B_1$ and $\pi_1 B_2$ are isomorphic, and as they are $K(G,1)$'s, there is a homotopy equivalence
\[
f: B_1 \to B_2.
\]
Let $g$ be the map 
\[
g = (f \times \id) \circ (p_1 \times p_2): E \to B_2 \times B_2.
\] 
By the above,
\[
\im(g)= \Delta = \{(x,x) \mid x \in B_2\}.
\]
Being non-surjective, $g$ has degree $0$. As $p_1 \times p_2$ is the composition of $g$ with a homotopy equivalence, we conclude that also $\deg(p_1 \times p_2) = 0$.
\end{proof}

In general the condition $\deg (p_1 \times p_2) = 0$ on a bi-projection does not imply that the associated fiberings are equivalent. However, in the setting of the Johnson kernel, this is indeed the case.
\begin{proposition} \label{proposition:jfiber}
Suppose $E$ is the total space of a surface bundle over a surface in two ways: $p_1: E \to B_1$ and $p_2: E \to B_2$.  Let $F_1, F_2$ denote fibers of $p_1, p_2$ respectively. Suppose that $\rho_1: \pi_1 B_1 \to \Mod(F_1)$ is contained in the Johnson kernel $\K_g$. Then the following are equivalent:
\begin{enumerate}
\item The fiberings $p_1, p_2$ are not $\pi_1$-fiberwise diffeomorphic.
\item The fiber subgroups $\pi_1 F_1, \pi_1 F_2 \le \pi_1 E$ are distinct.
\item $\deg (p_1 \times p_2) \ne 0$.
\item $E$ is diffeomorphic to $B_1 \times B_2$.
\end{enumerate}
\end{proposition}

The additional assertions in Proposition \ref{proposition:jfiber} will be proved in the course of establishing Theorem \ref{theorem:double} (see Remark \ref{remark:wrapup}). \\

\section{Surface bundles over surfaces with unique fiberings}\label{section:single}

In this section, we prove Theorem \ref{theorem:single}. The additive structure of $H^* E$ is central to everything that follows in the paper, and so we begin with a review of the relevant results. The following theorem was formulated and proved by Morita in \cite{morita} for the case of field coefficients of characteristic not dividing $\chi(F)$; subsequently this was improved to integral coefficients in the cohomological setting by Cavicchioli, Hegenbarth and Repov\v{s} in \cite{CHR}.
\begin{proposition}[{\bf Morita, Cavicchioli - Hegenbarth - Repov\v{s}}]\label{proposition:morita}
The Serre spectral sequence (with twisted coefficients) of any surface bundle $F \to E \to B$ collapses at the $E_2$ page. Consequently, there are noncanonical isomorphisms for all $k$
\begin{align*}
H_k (E, \Q) &= H_k (B,\Q) \oplus H_{k-1}(B, H_1 (F, \Q)) \oplus H_{k-2} (B, \Q)\\
H^k (E, \Z) &= H^k (B, \Z) \oplus H^{k-1}(B, H^1 (F, \Z)) \oplus H^{k-2}(B,\Z)
\end{align*}
The $H_{k-2} B$ summand of $H_k E$ is canonical, and is realized by the Gysin map $p^!$ which associates to a homology class $x \in B$ the induced sub-bundle $E_x$ sitting over $x$. Similarly, the $H^k B$ summand is canonical via the pullback map $p^*: H^k B \to H^k E$.\\
\indent If $F \to E \to B$ has monodromy in $\mathcal I_g$, then the coefficient system is untwisted and $H^*(E,\Z) \approx H^*(B, \Z) \otimes H^*(F, \Z)$ additively. In particular, $H^*(E, \Z)$ is torsion free, and so by the universal coefficients theorem, there is also an isomorphism $H_*(E, \Z) \approx H_*(B, \Z) \otimes H_*(F, \Z)$. 
\end{proposition}
Because the surface bundles we will be considering in this paper have monodromy lying in $\mathcal{I}_g$, we will subsequently take all coefficients to be $\Z$ without further mention. A remark which is obvious from Proposition \ref{proposition:morita} is that if $*$ generates $H_0 B$, then $p^!(*)$ is a primitive class; we will use this fact later on. Here and throughout, we will use the notation 
\[
[F] = p^!(*) \in H_2 E
\]
to denote the (pushforward of the) fundamental class of the fiber.

The following result is a well-known application of the theory of the Gysin homomorphism, and we state it without proof.
\begin{proposition}\label{proposition:gysin}
Let $p: E \to B$ be a surface bundle with fiber $F$. If $\chi(F) \ne 0$, then there are injections
\begin{align*}
p^*: 	&H^*(B, \Q) \to H^*(E, \Q)\\
p^!:	&H_k(B, \Q) \to H_{k+2}(E, \Q).
\end{align*}
In the case where $H_*(E,\Z)$ is torsion-free, the same statements hold with $\Z$ coefficients. In particular, this is true whenever $E$ has monodromy lying in $\mathcal I_g$, since in this case $H^*(E, \Z)$ is isomorphic to $H^*(F, \Z) \otimes H^*(B, \Z)$ as an abelian group (see Proposition \ref{proposition:morita}).
\end{proposition}

For surface bundles over surfaces with multiple fiberings, there is an extension of the previous result.

\begin{lemma}\label{lemma:direct}
Let $E$ be a 4-manifold with two distinct surface bundle structures $p_1: E \to B_1$ and  $p_2: E \to B_2$. Then the intersection
\[
p_1^*(H^1(B_1, \Q)) \cap p_2^*(H^1(B_2, \Q)) = \{0\},
\]
and so by Proposition \ref{proposition:gysin}, there is a canonical injection 
\[
p_1^* \times p_2^*: H^1(B_1, \Q)) \oplus H^1(B_2, \Q) \into H^1(E, \Q).
\]
\end{lemma}
\begin{proof}
By the universal coefficients theorem, for any space $X$ there is an identification
\[
H^1(X, \Q) \approx \Hom(\pi_1 X, \Q).
\]
Under this identification, a character $\alpha \in \Hom(\pi_1 B_i, \Q)$ is pulled back to $p_i^*(\alpha) \in \Hom(\pi_1 E, \Q)$ by precomposition with $(p_i)_*$. In particular, $p_i^* (\alpha)$ vanishes on $\pi_1 F_i = \ker (p_i)_*$. Therefore, any character $\alpha \in p_1^*(H^1(B_1, \Q)) \cap p_2^*(H^1(B_2, \Q))$ must vanish on the subgroup generated by $(\pi_1 F_1)( \pi_1 F_2)$.\\
\indent By Lemma \ref{lemma:findex} below, $(\pi_1 F_1)( \pi_1 F_2)$ has finite index in $\pi_1 E$. For any group $\Gamma$, any character $\alpha: \Gamma \to \Q$ vanishing on a finite-index subgroup must vanish identically, proving the claim.
\end{proof}

\begin{lemma}\label{lemma:findex}
Let $E$ be a surface bundle over a surface with two distinct fiberings $p_i: E \to B_i$; let the fibers be denoted $F_1$ and $F_2$, respectively. Then $(\pi_1 F_1)( \pi_1 F_2)$ has finite index in $\pi_1 E$. 
\end{lemma}
\begin{proof}
Consider the cross-projection $\pi_1 F_1 \to \pi_1 B_2$. Let the image of $\pi_1 F_1$ in $\pi_1 B_2$ be denoted $\Gamma$. This is a finitely-generated normal subgroup of $\pi_1 B_2$. For any surface group of genus $g \ge 2$, any nontrivial finitely-generated normal subgroup has finite index (see Property $(\mathscr{D}6)$ in \cite{FEA2}). If $\Gamma$ is the trivial group, then $\pi_1 F_1 \le \pi_1 F_2$, necessarily again of finite index. In this case, the image of $\pi_1 F_2$ in $\pi_1 B_1$ is therefore finite, but $\pi_1 B_1$ is torsion-free. We conclude that $\Gamma \le \pi_1 B_2$ has finite index. The kernel of the map $\pi_1 E \to (\pi_1 B_2/\Gamma)$ is exactly $(\pi_1 F_1)( \pi_1 F_2)$.
\end{proof}

\begin{theorem}\label{theorem:single} 
Let $p: E \to B$ be a surface bundle over a surface $B$ of genus $g \ge 2$ with monodromy representation $\rho: \pi_1 B \to \Mod(\Sigma_g)$. Suppose that the space of invariant cohomology $(H^1(F,\Q))^\rho$ (equivalently, the coinvariant homology of the fiber $(H_1(F,\Q))_\rho$) vanishes. Then $E$ admits a unique fibering.
\end{theorem}

\begin{proof}
For any surface bundle $p: E \to B$ with monodromy $\rho$ and any choice of coefficients, there is a (noncanonical) splitting
\[
H^1 E = p^*(H^1 B) \oplus (H^1 F)^\rho.
\]
(see Proposition \ref{proposition:morita}). If $(H^1(F,\Q))^\rho = 0$, then this reduces to
\[
H^1(E, \Q) = p^* H^1(B,\Q).
\]
If $p_2:E \to B_2$ is a second, distinct fibering, the above shows that $p_2^*(H^1(B_2, \Q)) \le  p^* H^1(B,\Q)$. However, this contradicts Lemma \ref{lemma:direct}. \end{proof}

\section{Bi-projections} \label{section:biproj}

In this section we state and prove the key lemma from differential topology needed for the proof of Theorem \ref{theorem:double}. 
\begin{proposition}\label{proposition:degree}
Let $E$ be a 4-manifold with surface bundle structures $p_1: E \to B_1$ and $p_2: E \to B_2$. Let $F_1, F_2$ denote fibers of $p_1, p_2$ lying over a regular value of $p_1 \times p_2$. If $\deg (p_1 \times p_2: E \to B_1 \times B_2) \ne 0$, then the following five quantities are equal:

\begin{packed_enum}
\item $\deg (p_1 \times p_2: E \to B_1 \times B_2)$
\item $\deg (p_1|_{F_2}: F_2 \to B_1)$
\item $\deg (p_2|_{F_1}: F_1 \to B_2)$
\item $I_E(F_1,F_2)$ (the algebraic intersection number)
\item $|F \cap F_2|$ (the cardinality of the intersection).
\end{packed_enum}
As (5) indicates, this quantity is always positive.
\end{proposition}

\begin{proof}
As $p_1$ and $p_2$ are projection maps for fiber bundle structures on $E$, they are everywhere regular, and $\ker (dp_1)_x$ is identified with the tangent space to the fiber of $p_1$ through $x$. Let $z = (b_1, b_2) \in B_1 \times B_2$ be a regular value for $p_1 \times p_2$. It follows from the assumption that $\deg (p_1 \times p_2: E \to B_1 \times B_2) \ne 0$ that $d (p_1 \times p_2)_x$ is an isomorphism for all $x \in (p_1 \times p_2)^{-1}(z)$ (and that this preimage is non-empty). The kernel of $d (p_1 \times p_2)_x$ is just the intersection of the kernels of $d(p_1)_x$ and $d(p_2)_x$. It follows that for all $x \in (p_1 \times p_2)^{-1}(z)$,
\begin{equation}\label{equation:decomp}
T_x E \approx T_x F_1 \oplus T_x F_2.
\end{equation}
Note that this shows that the fibers $F_1,F_2$ over $b_1,b_2$ respectively are transverse. \\
\indent Choose orientations for $E, B_1, B_2$. This specifies an orientation on each fiber of $p_1$ and $p_2$ via the following decomposition, where $H_x$ is any complement to $T_x F_1 = \ker d(p_1)_x$:
\[
T_x F_1 \oplus H_x \approx T_x E.
\]
The orientation on $H_x$ is specified by the isomorphism $H_x \approx T_{p_1(x)}B_1$. Of course an analogous convention orients each fiber of $p_2$. In particular, it follows from (\ref{equation:decomp}) that at any regular point for $p_1\times p_2$, we can take $H_x = T_x F_2$, and that the restriction of $d(p_1)_x$ to $T_x F_2$ is an isomorphism.\\

Recall that if $f: X^n\to Y^n$ is a smooth map of oriented closed $n$-manifolds, then 
\[
\deg(f) = \sum_{x \in f^{-1}(y)} \epsilon(x),
\]
where $y$ is any regular value of $f$, and $\epsilon(x) = 1$ if the orientation on $T_y Y$ induced by $df_x$ agrees with the pre-chosen orientation on $Y$, and $\epsilon(x) = -1$ otherwise. If $Y, Z$ are smoothly embedded and transversely intersecting oriented submanifolds of the oriented manifold $X$ such that $\dim(X) = \dim(Y) + \dim(Z)$, then the algebraic intersection number of $Y$ and $Z$ is computed as
\[
I_X(Y,Z) = \sum_{w \in Y \cap Z} \epsilon(w),
\]
where $\epsilon(w) = 1$ if the orientation on $T_w X$ given by $T_w Y \oplus T_w Z$ agrees with the pre-chosen orientation on $X$, and $\epsilon(w) = -1$ otherwise.

It follows from the definitions that
\[
(p_1 \times p_2)^{-1}(b_1,b_2) = p_1|_{F_2}^{-1}(b_1) = {p_2}|_{F_1}^{-1}(b_2) =  F_1 \cap F_2.
\]
Therefore each of the sums computing $(\ref{proposition:degree}.1) - (\ref{proposition:degree}.5)$ take place over the same set of points. So it remains only to show that in each of the contexts $(\ref{proposition:degree}.1)-(\ref{proposition:degree}.4)$, the relevant orientation convention assigns a positive value.
 
The orientation number assigned to $x \in (p_1 \times p_2)^{-1}(b_1,b_2)$ is given by the sign of the determinant of the map
\[
d(p_1 \times p_2)_x: T_x E \to T_{b_1} B_1 \oplus T_{b_2}B_2.
\]
By the above discussion, our orientation convention stipulates that 
\[
d ({p_1|_{F_2}})_x: T_x F_2 \to T_{b_1} B_1
\]
is an orientation-preserving isomorphism, and similarly for $d(p_2|_{F_1})$. This proves the equality of $(\ref{proposition:degree}.2)$ and $(\ref{proposition:degree}.3)$ with $(\ref{proposition:degree}.5)$.\\
\indent As
\[
T_x F_1 = \ker d (p_1)_x \qquad \text{ and } \qquad T_x F_2 = \ker d (p_2)_x
\]
it follows that $d(p_1 \times p_2)_x$ has a block-diagonal decomposition
\[
d(p_1 \times p_2)_x = d (p_1)_x \oplus d (p_2)_x : T_x F_1 \oplus T_x F_2 \to T_{b_2} B_2 \oplus T_{b_1} B_1,
\]
from which it follows that $x$ also carries a positive orientation number in setting $(\ref{proposition:degree}.1)$. Finally, the orientation number for $x$ as a point of intersection between $F_1$ and $F_2$ records whether the orientations of $T_x E$ and $T_x F_1 \oplus T_x F_2$ agree, but we have already seen that they necessarily do.
\end{proof}


\section{Cup products and the Johnson homomorphism}\label{section:johnson}

\indent The goal of this section is to give a construction of embedded submanifolds in a surface bundle over a surface $E$ that will be explicit enough to compute the intersection form on homology, or dually the cup product structure in cohomology. One of the original definitions of the Johnson homomorphism was via the cup product structure in surface bundles over $S^1$. In this section we turn this perspective on its head and explain how the Johnson homomorphism computes the cup product structure in a surface bundle over a surface (in fact, these methods extend to surface bundles over arbitrary manifolds). The submanifolds we construct will be codimension-$1$ (i.e. $3$-manifolds), and built so that their intersection theory is explicitly connected to the Johnson homomorphism.

To this end, in Section \ref{subsection:intjohnson} we give a discussion of the definition of the Johnson homomorphism in the setting of the cup product in surface bundles over $S^1$. The centerpiece of this is the construction of geometric representatives for classes in $H^1$, via embedded surfaces which we call ``tube-and-cap surfaces''. Then in Section \ref{subsection:construction}, we return to the original problem of constructing representatives for classes in $H^1$ of a surface bundle over a surface as embedded $3$-manifolds. The construction is carried out so that the intersection of particular pairs of these $3$-manifolds is a tube-and-cap surface, thereby realizing the link between cup products in surface bundles over surfaces and the Johnson homomorphism. 

\subsection{From the intersection form to the Johnson homomorphism, and back again}\label{subsection:intjohnson} 
In this subsection we will begin to dive into the theory of the Torelli group in earnest, so we begin with a brief review of the relevant defintions. The {\it Torelli group} $\mathcal I_g$ is the kernel of the symplectic representation $\Psi: \Mod(\Sigma_g) \to \Sp_{2g}(\Z)$. The {\it Johnson kernel} $\mathcal K_g$ is the subgroup of $\mathcal I_g$ generated by all Dehn twists $T_\gamma$ about {\it separating} curves $\gamma$. It is a deep theorem of D. Johnson that $\mathcal K_g$ can alternately be characterized as the kernel of the Johnson homomorphism $\tau$ to be defined below.\\

Let $\phi \in \mathcal{I}_g$ be a Torelli mapping class, and build the mapping torus $M_\phi = \Sigma_g \times I / \{(x,1) \sim (\phi(x),0)\}$. As $\phi \in \mathcal{I}_g$ for any curve $\gamma \subset \Sigma_g$, the homology class $[\gamma] - \phi_*[\gamma]$ is zero. Thus there exists a map of a surface $i: S \to \Sigma_g$ which cobounds $\gamma \cup \phi(\gamma)$. Indeed, there exists an {\em embedded} surface $S \le \Sigma_g \times I$ whose boundary is given by
\[
\partial S = \gamma \times \{1\} \cup \phi(\gamma) \times \{0\}.
\]
To see this, recall that since $S^1$ is a $K(\Z,1)$, there is a correspondence 
\[
H^1(\Sigma_g, \Z) \approx [\Sigma_g, S^1].
\]
Via Poincar\'e duality, 
\[
H^1(\Sigma_g, \Z) \approx H_1(\Sigma_g, \Z).
\]
The induced correspondence
\[
H_1(\Sigma_g, \Z) \approx [\Sigma_g, S^1]
\]
is realized by taking the preimage of a regular value, which will be an embedded submanifold. Under this correspondence, homotopic maps $f,g: \Sigma_g \to S^1$ yield homologous submanifolds, and conversely. Therefore, the maps $f,g: \Sigma_g \to S^1$ which determine $\gamma, \phi(\gamma)$ are homotopic. This gives the desired map $F: \Sigma_g \times I \to S^1$ such that the preimage of a regular value is an embedded surface $S$ cobounding $\gamma$ and $\phi(\gamma)$. 

In fact, the choice of $S$ is not unique. Let $i': S' \to M_\phi$ be any map of a {\em closed} surface to $M_\phi$. Then the chain $S + S'$ satisfies $\partial (S + S') = \partial S = \gamma - \phi(\gamma)$. Nonetheless, given any $S$ satisfying $\partial(S) = \gamma - \phi(\gamma)$, we can form a closed submanifold of $M_\phi$ in the following way. We begin with a tube, diffeomorphic to $S^1 \times I$, embedded into $M_\phi$ as $\phi(\gamma) \times [0,1/3] \cup \gamma \times [2/3,1]$. We may then glue in $S$ to $\Sigma_g \times [1/3, 2/3]$. The result is a smoothly-embedded oriented submanifold $\Sigma_\gamma \subset M_\phi$, which will descend to a homology class $\Sigma_z$ (here $z = [\gamma]$). See Figure \ref{figure:tube}.
\begin{center}
\begin{figure}
\includegraphics{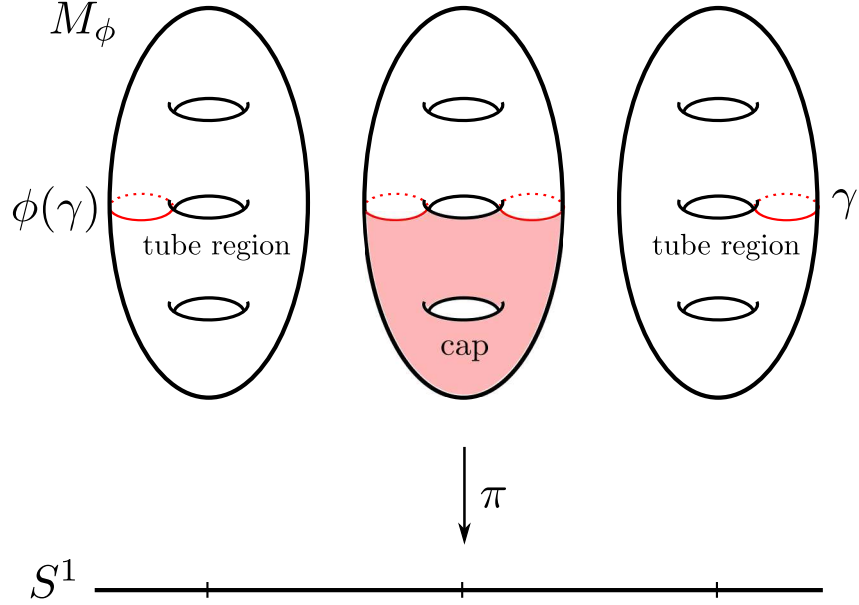}
\caption{A tube surface}
\label{figure:tube}
\end{figure}
\end{center}

For convenience, we introduce the following terminology for these surfaces, which we will refer to as {\em tube surfaces}. The {\em tube} of a tube surface is the cylinder $S^1 \times I = \phi(\gamma) \times [0,1/3] \cup \gamma \times [2/3,1]$, and the {\em cap} is the subsurface $S$.

We assign an orientation to $\Sigma_{\gamma}$ as follows. The tangent space to a point $x$ contained in the tube has a direct sum decomposition via
\begin{equation}\label{eqn:star}
T_x \Sigma_\gamma = V \oplus T_x\gamma,
\end{equation}
where $V$ is any preimage of $T_{\pi(x)}S^1$ and $T_x \gamma$ is interpreted as the tangent space to the copy of $\gamma$ sitting in the fiber containing $x$. Both of the summands in (\ref{eqn:star}) have orientations induced from those on $S^1$ and $\gamma$ respectively, and this endows $T_x\Sigma$ with an orientation. This can then be extended over the cap surface in a coherent way, since $S$ was chosen to be a boundary for $[\gamma] - [\phi(\gamma)]$ with $\Z$ coefficients. 

Recall however that the choice of $S$ was not unique. Any closed surface mapping into $\Sigma_g$ is homologous to some multiple of the fundamental class, and so the above procedure really defines a homomorphism $H_1 \Sigma_g \to H_2 M_\phi / [F]$, where $[F]$ is the fundamental class of the fiber. If the bundle has a section $\sigma: S^1 \to M_\phi$, then we can choose $S$ so that $\im \sigma$ and $\Sigma_z$ have zero algebraic intersection, which gives a canonical lift $H_1 \Sigma_g \to H_2 M_\phi$. In the absence of such auxiliary data, we instead just choose an {\em arbitrary} lift, and we will account for the consequences later.\\
\indent Having chosen an embedding $i: H_1 \Sigma_g \into H_2 M_\phi$ such that $z \mapsto \Sigma_z$, there is an associated direct sum decomposition of $H_2 M_\phi$, namely
\[
H_2 M_\phi = \pair{[F]} \oplus \im i.
\]
Relative to such an embedding, we form the map $\tau(\phi) \in \Hom(\wedge^3 H_1 \Sigma_g, \Z)$ by
\[
\tau(\phi)(x \wedge y \wedge z) = \Sigma_x \cdot \Sigma_y \cdot \Sigma_z,
\]
the term on the right being interpreted as the triple algebraic intersection of the given homology classes. Suppose a section exists, and that the $\Sigma_x$ have been constructed accordingly. In this case, D. Johnson showed that the map 
\begin{align*}
\tau: \mathcal{I}_{g,*} &\to \Hom(\wedge^3 H_1 \Sigma_g, \Z)\\
\phi &\mapsto \tau(\phi)
\end{align*}
is a surjective homomorphism. See \cite[Chapter 6]{FM} for a summary of the Johnson homomorphism, including two alternative definitions. The (pointed) {\em Johnson kernel} $\K_{g,*}$ is defined to be the kernel of $\tau$; in our context this exactly means that all triple intersections between the various $\Sigma_x$ vanish.

Having fixed a family of $\Sigma_x$, it is then easy to compute the entire intersection form on $\wedge^3 H_2 M_\phi$. Certainly $[F]^2 = 0$. It is also fairly easy to see that
\[
[F] \cdot \Sigma_x \cdot \Sigma_y = i(x,y),
\]
where $i(x,y)$ denotes the algebraic intersection pairing in $H_1 \Sigma_g$. Indeed, by picking the choice of fiber to intersect $\Sigma_x$ on the tube, it is clear that the result is simply the curve $x$, so that $[F] \cdot \Sigma_x \cdot \Sigma_y$ computes the intersection of $x,y$ on $F$, at least up to a sign that may be introduced by the (non)compatibilities of the various orientation conventions in play. A quick check reveals this sign 	to be positive.\\
\indent We will now be able to account for the ambiguity introduced by our choice of embedding $i: H_1 \Sigma_g \into H_2 M_\phi$, which will in turn lead to the definition of the Johnson homomorphism on the closed Torelli group $\mathcal I_g$. Suppose that $\Sigma_w' = \Sigma_w + k_w[F]$ is some other set of choices which is {\em coherent} in the sense that $\Sigma'_w + \Sigma'_z = \Sigma'_{w + z}$ (i.e. $x \mapsto k_x \in H^1 \Sigma_g$). By linearity, 
\begin{align*}
\Sigma'_x \cdot \Sigma'_y \cdot \Sigma'_z &= \Sigma_x \cdot \Sigma_y \cdot \Sigma_z + k_x i(y,z) + k_y i(z,x) + k_z i(x,y)\\
&= \tau(\phi)(x \wedge y \wedge z) + k_x i(y,z) + k_y i(z,x) + k_z i(x,y)\\
&= \tau(\phi)(x \wedge y \wedge z) + C^*(k);
\end{align*}
here $C: \wedge^3 H_1 \Sigma_g \to H_1 \Sigma_g$ is the contraction with the symplectic form $i(\cdot, \cdot)$, and $k \in \Hom(H_1 \Sigma_g, \Z)$ is the form such that $k(w) = k_w$. The upshot of this calculation is that $\tau(\phi)$ is well-defined as an element of $\Hom(\wedge^3 H_1 \Sigma_g, \Z) / \im C^*$, which can be identified with the more familiar space $\wedge^3 H / H$ (here we adopt the usual convention that $H = H_1 \Sigma_g$). The  {\em Johnson homomorphism} on the closed Torelli group is then defined via
\begin{align*}
\tau: \mathcal{I}_g &\to \Hom(\wedge^3 H_1 \Sigma_g, \Z) / \im C^* \approx \wedge^3 H / H\\
\phi &\mapsto \tau(\phi).
\end{align*}
As before, the (closed) {\em Johnson kernel} $\K_g$ is the kernel of $\tau$. As mentioned above, work of D. Johnson shows that there is an alternative characterization via
\[
\K_g = \pair{T_\gamma \mid \gamma \mbox{ separating scc}}.
\] 

\begin{remark} The construction given above with the tube-cap surfaces is a concrete realization of the isomorphism $H_1 \Sigma_g \approx H_2 M_\phi / [F]$ coming from the Serre spectral sequence for $p: M_\phi \to S^1$. In fact, this same construction will work for an arbitrary $\phi \in \Mod_g$, yielding an isomorphism $(H_1 \Sigma_g)^\phi \approx H_2 M_\phi / [F]$, but we do not pursue this here.
\end{remark}

\indent The above discussion shows how to construct the Johnson homomorphism in terms of the intersection form on $M_\phi$. Conversely, we will show next how to reconstruct the intersection form on $M_\phi$ from the data of the Johnson homomorphism $\tau(\phi) \in \wedge^3 H / H \approx \Hom(\wedge^3 H \Sigma_g, \Z) / \im C^* $. Begin by selecting an {\em arbitrary} lift $\tilde \tau (\phi)$ of $\tau(\phi)$ (of course, the presence of a section gives a canonical such choice). Next, construct a coherent family of homology classes $\Sigma'_x$ by making choices arbitrarily. Define $\tau'(\phi) \in \Hom(\wedge^3 H, \Z)$ by 
\[
\tau'(\phi)(x \wedge y \wedge z) = \Sigma'_x \cdot \Sigma'_y \cdot \Sigma'_z.
\]

\indent There is no reason to suspect that $\tau'(\phi) = \tilde \tau (\phi)$. However, as we saw above, we do know that $\tau'(\phi) - \tilde \tau (\phi) \in \im C^*$, and so there is some functional $\alpha \in H^1 \Sigma_g$ such that $\tau'(\phi) - \tilde \tau (\phi) = C^*(\alpha)$. This functional $\alpha$ will allow us to choose the correct set of $\Sigma_x$ so that the triple intersections are computed by our choice of $\tilde \tau (\phi)$. 

\begin{lemma}\label{proposition:derp}
We assume the notation of the above setting. By taking
\[
\Sigma_x = \Sigma'_x - \alpha(x)[F], 
\]
there is an equality for all $x,y,z$:
\[
\Sigma_x \cdot \Sigma_y \cdot \Sigma_z = \tilde \tau (\phi)(x \wedge y \wedge z).
\]
\end{lemma}
\begin{proof}
Compute:
\begin{align*}
\Sigma_x \cdot \Sigma_y \cdot \Sigma_z &= \Sigma'_x \cdot \Sigma'_y \cdot \Sigma'_z - \alpha(x) i(y,z) - \alpha(y) i(z,x) - \alpha(z) i(x,y)\\
&= \tau'(\phi)(x \wedge y \wedge z) - C^*(\alpha)(x \wedge y \wedge z)\\
&= \tilde \tau (\phi).
\end{align*}
\end{proof}

\subsection{Intersections in surface bundles over surfaces, and beyond}\label{subsection:construction} The methods of the previous subsection can be adapted to give a description of certain cup products in $H^1 E$, where $p: E^{n+2} \to B^n$ has monodromy lying in $\mathcal I_g$. The idea will be to define an embedding, as before,
\[
i: H_1 \Sigma_g \into H_{n+1} E,
\]
by constructing submanifolds $M_\gamma$ for curves $\gamma \subset \Sigma_g$ by means of a higher-dimensional ``tubing construction''. Then the triple intersections of collections of $M_x$ will be partially computable via the Johnson homomorphism in a certain sense to be described below. In this subsection we will first briefly sketch the properties we require of the submanifolds $M_\gamma$, then we will give the construction. Then in Section \ref{subsection:intform}, we will determine much of the intersection pairing in $H_*(E, \Z)$.\\
\indent Our construction will provide, for each simple closed curve $\gamma \subset F$, a submanifold $M_\gamma$, such that if $[\gamma] = [\gamma']$, then also $[M_\gamma] = [M_{\gamma'}]$. If $[\gamma] = x$, we write $M_x$ in place of $[M_\gamma]$. Let $p: E \to B$ be a surface bundle with monodromy in $\mathcal I_g$, and let $\rho: \pi_1 B \to \mathcal{I}_g$ be the monodromy. By post-composing with $\tau: \mathcal I_g \to \wedge^3 H / H$, we obtain a map from $\pi_1 B$ to an abelian group, and so $\tau \circ \rho$ factors through $H_1 B$. By an abuse of notation we will write $\tau(b)$ for $b \in H_1 B$. \\
\indent This map computes (most of) the intersection form in $H_*(E)$. Recall the notation from Proposition \ref{proposition:morita}: given a curve $\alpha \subset B$, there is an induced bundle $E_\alpha$ over $\alpha$, which determines a homology class $E_a$. A given $M_\gamma$ can be intersected with $E_\alpha$ to yield a surface $\Sigma_{\alpha, \gamma}$ inside $E_\alpha$. Our construction will be set up so that
\[
M_x \cdot M_y \cdot M_z \cdot E_b = \tau(b)(x \wedge y \wedge z),
\]
possibly up to a sign. This is the sense in which $M_x \cdot M_y \cdot M_z$ is partially computable. As a remark, the intersections $M_x \cdot M_y \cdot M_z \cdot X$ for arbitrary $X \in H_3 E$ will all involve intersections with further $M_w$, and are describable (at least in the case of bundles with section) in terms of the higher Johnson invariants
\[
\tau: H_i(\mathcal I_{g,*}) \to \wedge^{i+2} H,
\]
but we will not pursue this point of view further in this paper.\\

\para{The construction} As usual, let $\pi: E \to B$ be a surface bundle with monodromy lying in $\mathcal I_g$ and associated Johnson map $\tau: H_1 B \to \wedge^3 H / H$. We turn now to the question of constructing suitable homology classes $M_x \in H_{n+1} E$, for $x \in H_1 \Sigma_g$. As the only case of interest in the present paper is where $B$ is a surface, we will content ourselves with describing the case when $M_\gamma$ is a 3-manifold. The reader may find it helpful to consult Figure \ref{figure:m} as they read this subsection.\\
\indent Consider a cell decomposition
\[
B = B^0 \subset B^1 \subset B^2
\]
of $B$, where $B^0$ consists of the single point $p$, there are $2g$ one-cells $\{a_1, b_1, \dots, a_h, b_h\}$, and a single two-cell $D$. For each one-cell $e$, there is an associated element of the monodromy, $\rho(e)$, such that the effect of transporting a curve $\gamma$ across $e$ (from the negative to the positive side, relative to orientations of $B$ and $e$) sends the isotopy class of $\gamma$ to $\rho(e) \gamma$. For a one-cell $e$, let $N(e) \approx e \times I$ be a (closed) regular neighborhood in $B$. We also let $N(p)$ be a small closed neighborhood of $p$. If necessary, shrink the $N(e)$ so that
\[
N:= N(a_1) \cup \dots \cup N(b_h) \setminus N(p)
\]
is a union of $2h$ disjoint rectangles.\\
\indent Let $\gamma \subset F$ be a simple closed curve on a fiber $F$ over a point in 
\[
D' := \overline{D \setminus \left(N(p) \cup N(a_1) \cup \dots \cup N(b_h) \right)}.
\]
By construction, $D'$ is nothing more than a closed disk (in the upper-left portion of Figure \ref{figure:m}, $D'$ is the closure of the complement of the shaded regions). The submanifold $M_\gamma$ will be constructed in three stages: first over $D'$, then over $N$, and finally over $N(p)$. Choose a trivialization $\pi^{-1}(D') \approx D' \times F$, and take $M_\gamma^1 = \gamma \times D'$ relative to this trivialization. Then $\partial(M_\gamma^1) \subset \pi^{-1}(\partial D')$. We specify an orientation on $M_\gamma^1$ as follows: a point $x \in M_\gamma^1$ has a decomposition of the tangent space via
\begin{equation}\label{eqn:starstar}
T_x M_\gamma^1 \approx T_{\pi(x)} B \oplus T_x \gamma.
\end{equation}
Both of these two summands carry pre-existing orientations, and $M_\gamma^1$ is then oriented by specifying the above isomorphism to be orientation-preserving. By analogy with the construction of tube surfaces, we refer to $M_\gamma^1$ as the {\em tube region} of $M_\gamma$.

Next we construct $M_\gamma^2$. Let $e$ be a one-cell, and consider the intersection $M_\gamma^1 \cap \pi^{-1} (N(e) \cap N)$. The base space $N(e) \cap N$ is just a rectangle, and so the bundle $\pi^{-1}(N(e) \cap N)$ is trivializable. We can therefore find a diffeomorphism
\[
\psi: \pi^{-1}(N(e) \cap N) \approx I \times I \times \Sigma_g
\]
under which $M_\gamma^1 \cap \pi^{-1} (N(e) \cap N)$ is identified with 
\[
\left(I \times \{0\} \times \gamma \right) \cup \left(I \times \{1\} \times \gamma' \right),
\]
where $\gamma'$ is some curve in the isotopy class of $\rho(e)(\gamma)$. As we saw in the previous subsection, for each $e$ there exists a family of embedded surfaces $S_e$ in $I \times \Sigma_g$ such that $\partial S_e = \{0\} \times \gamma \cup \{1\} \times \gamma'$. As before, make an arbitrary choice of $S_e$ for each $e$. We can then fill in $\pi^{-1} (N(e) \cap N)$ with $I \times S_e$ for each $e$, creating $M_\gamma^2$. As in the case of a tube surface, the orientation for $M_\gamma^1$ can be extended over each of these pieces coherently. We refer to $M_\gamma^2 \setminus M_\gamma^1$ as the {\em cap region} of $M_\gamma$.

It therefore remains to construct $M_\gamma^3 = M_\gamma$. The boundary of $M_\gamma^2$ lies on $\pi^{-1}(\partial N(p))$. We would like to be able to fill this boundary in by inserting a ``plug'' contained in $\pi^{-1}(N(p))$. {\em A priori}, there is a homological obstruction to this: if $[\partial M_\gamma^2] \ne 0$ in $H_2( \pi^{-1}(N(p)))$ then this problem is not solvable even on the chain level. However, this obstruction vanishes. To see this, observe that $\pi^{-1}(N(p))$ deformation-retracts onto $F = \pi^{-1}(p)$ (the bundle over $N(p)$ being trivial). In particular, a surface $S \subset \pi^{-1}(N(p))$ represents zero in $H_2(\pi^{-1}(N(p)))$ if and only if the degree of the map $S \to F$ coming from the retraction has degree zero.

The surface $\partial M_\gamma^2$ maps with degree zero onto $F$. To see this, we will examine our construction more closely. If $B$ is a surface of genus $h$, the one-cells of our construction can be labeled $a_1,b_1, \dots, a_h,b_h$. The picture near $p$ is therefore of $4g$ edge segments, one ingoing and one outgoing for each one-cell, ordered (counter-clockwise, say) as $a_1, b_1, a_1^{-1}, b_1^{-1}, \dots, a_h^{-1}, b_h^{-1}$ (where $e$ denotes an outgoing edge and $e^{-1}$ denotes an ingoing edge).

The surface $\partial M_\gamma^2$ is constructed as a union of $4g$ cap surfaces, one sitting over each ingoing and outgoing edge segment near $p$ (see the lower-right portion of Figure \ref{figure:m}). Recall that the monodromy of the bundle acts on the fiber as follows: the effect of crossing positively over a one-cell $e$ is to act by $\rho(e)$. Therefore, the cap surface for a curve $\eta$ sitting over an edge $e^{\pm1}$ connects $\eta$ to $\rho(e)^{\pm 1} \eta$. We denote this surface by $\Sigma_{\rho(e)^{\pm 1}, \eta}$. For example, the cap surface sitting over $a_1$ is $\Sigma_{\rho(a_1), \gamma}$. The cap surface over $a_1^{-1}$ is obtained from this by transporting $\Sigma_{\rho(a_1), \gamma}$ along $a_1$. 

Let $\alpha: I \to B$ be a parameterization of $a_1$, with $\alpha(0) = \alpha(1) = p$. Consider the product $\Sigma_{\rho(a_1), \gamma} \times [\epsilon, 1 - \epsilon]$, embedded in $\pi^{-1}(N(e) \setminus N)$ by transporting $\Sigma_{\rho(a_1), \gamma}$ along $a_1$. Relative to a fixed identification of $\pi^{-1}(N(p))$ with $N(p) \times \Sigma_g$, the endpoint $\Sigma_{\rho(a_1), \gamma} \times \epsilon$ is diffeomorphic to $\Sigma_{\rho(a_1), \gamma}$, while the endpoint $\Sigma_{\rho(a_1), \gamma} \times (1 - \epsilon)$ is diffeomorphic to $\rho(b_1)(\Sigma_{\rho(a_1), \gamma})$ with the orientation reversed. Generally, there are $2g$ distinct cap surfaces, one for each edge $a_i, b_i$, and each appears twice (modulo the action of the monodromy), with opposite orientations. Therefore, when projecting onto $F$, the total degree is zero, and so there exists a 3-chain $M_p$ in $\pi^{-1}(N(p))$  satisfying 
\[
\partial M_p = \partial M_\gamma^2.
\]

\indent The last thing to do is to explain why $M_p$ can be replaced with a smooth 3-manifold. This will follow from general results on representing (relative) codimension-one homology classes by smooth submanifolds (with boundary). The argument proceeds along very similar lines to the construction of embedded cap surfaces in fibered 3-manifolds described above. For an oriented manifold $X$ with boundary, Lefschetz duality gives an isomorphism
\[
H_{n-1}(X,\partial X, \Z) \approx H^1(X, \Z) \approx [X, S^1]
\]
\indent In our setting, the surface $\partial M_\gamma^2$ is represented by a map 
\[
f: \pi^{-1}(\partial N(p)) \to S^1,
\]
such that $\partial M_\gamma^2 = f^{-1}(*)$ for some regular value $* \in S^1$. Similarly, the homology class of $M_p$ in $H_3(\pi^{-1}(N(p)))$ corresponds to a map 
\[
F: \pi^{-1}(N(p)) \to S^1.
\]
Moreover, as $\partial M_p = \partial M_\gamma^2$, they represent the same homology class in $H_2 (\pi^{-1}(\partial N(p)), \Z)$. This means that the maps $f$ and $F\mid_{\pi^{-1}(\partial N(p))}$ are homotopic. We can therefore concatenate this homotopy with $F$, to obtain a map 
\[
\tilde F: \pi^{-1} (N(p)) \to S^1.
\]
On the boundary, $\tilde F = f$, and is therefore transverse to $* \subset S^1$. In order to replace $M_p$ by a smooth submanifold such that $\partial M_p = \partial M_\gamma^2$, we must therefore perturb $\tilde F$ away from a neighborhood of $\pi^{-1}(\partial(N(P)))$ and make the result everywhere transverse to $* \subset S^1$. The Extension Theorem (see \cite[p. 72]{gp}) asserts that we can do precisely this, and by gluing the boundaries of this new submanifold and $M_\gamma^2$, we have succeeded in constructing the closed submanifold $M_\gamma$. We refer to this portion of $M_\gamma$ as the {\em plug}. Lastly we extend the orientation on $M_\gamma^2$ over all of $M_\gamma$. \\
\begin{center}
\begin{figure}
\includegraphics{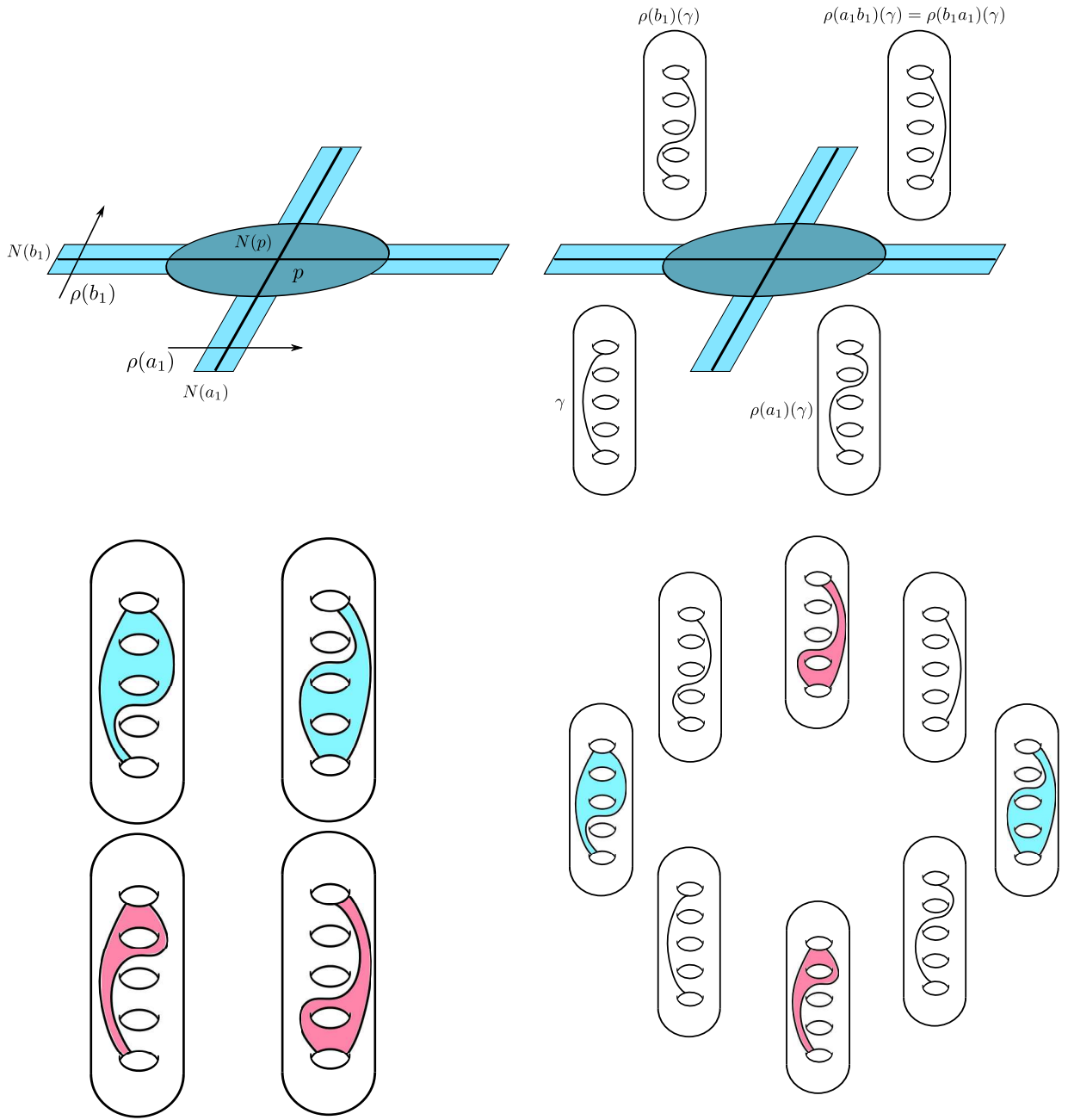}
\caption{Upper left: The neighborhoods $N(e)$ and $N(p)$. Upper right: $M_\gamma^1$ intersected with four different fibers. Lower left: Cap surfaces, lying over different portions of $N$. Lower right: A depiction of $M_\gamma^2 \cap \pi^{-1}(\partial N)$.}
\label{figure:m}
\end{figure}
\end{center}

\indent An essential feature of the above construction is the relationship between an $M_\gamma$ and a sub-bundle $E_\alpha$ lying over a curve $\alpha \subset B$. Suppose $\alpha$ is chosen so that relative to the cell decomposition of $B$ used in constructing $M_\gamma$, $\alpha$ is transverse to all the one-cells $e$, and does not pass through $N(p)$. Then a little visual imagination reveals that the intersection of $M_\gamma$ and $E_\alpha$ is given by a tube surface for $\gamma$ sitting inside $E_\alpha$. We call the resulting surface $\Sigma_{\alpha, \gamma}$, and then $[\Sigma_{\alpha, \gamma}]$ is denoted by $\Sigma_{a,x}$, where $[\alpha] = a$ and $[\gamma] = x$.

We define a {\em family} of $M_x$ to be a set of $M_x$ for each $x \in H_1 F$ such that for all $c \in \Z$ and $x,y \in H_1 F$,
\[
M_{c x + y} = c M_x + M_y.
\]
Different choices of $M_x$ lead to different spaces of $\Sigma_{b,x}$, but conversely, a choice of a family of $M_x$ leads to a corresponding distinguished summand of $H_2 E$.\\

\subsection{Determination of the intersection form}\label{subsection:intform} It remains to give a description of the cup product structure on $H^*(E, \Z)$; equivalently, we will describe the intersection form. By Poincar\'e duality, it suffices to determine, for each $X$, the set of pairings $X \cdot Y$.
\begin{proposition}\label{proposition:cups}
\begin{packed_enum} Let $i_B$ and $i_F$ denote the algebraic intersection pairing on the homology of the base and on the fiber, respectively. 
\item There exists a unique class $C \in H_2 E$ such that $C \cdot \Sigma_{b,z}= 0$ for all $b \in H_1 B $, $z \in H_1 \Sigma_g$, and $C \cdot [F] = 1$. The intersection pairing $H_2 E \otimes H_2 E \to \Z$ is given as follows, where $e = C^2$ by definition.
\[
\begin{array}{|c|c|c|c|} \hline
					&	C	&	[F]	&	\Sigma_{a,z}	\\ \hline
C					&	e	&	1	&	0					\\ \hline
[F]					&	1	&	0	&	0					\\ \hline
\Sigma_{b,w}			&	0	&	0	&	-i_B(a,b)i_F(z,w)		\\ \hline
\end{array}
\]
In the case where the monodromy is contained in the Johnson kernel, we have $e$ = 0.

\item For any family of $M_x$, we have
\begin{align*}
E_a \cdot E_b &= i_B(a,b) [F]\\
M_x \cdot E_b &= \Sigma_{b,x}\\
M_z \cdot M_w \cdot [F] &= i_F(z,w).
\end{align*}

\item For a given lift of $\tau: H_1 B \to \wedge^3 H / H$ to $\tilde \tau: H_1 B \to \wedge^3$, there exists a splitting 
\[
H_3 E = \pi^!(H_1 B) \oplus H_1 M = \{E_b, b \in H_1 B\} \oplus \{M_z, z \in H_1 F\}
\]
relative to which
\begin{align*}
&M_x \cdot M_y \cdot M_z \cdot E_b= M_x \cdot M_y \cdot \Sigma_{b,z} = \tilde\tau(b)(x \wedge y \wedge z).
\end{align*}
In the case where the monodromy is contained in the Johnson kernel, we can take the canonical lift to zero, and for this family of $M_x$ we have 
\begin{align*}
C \cdot M_x &= 0\\
C^2 &= 0
\end{align*}
for all $x \in H_1 \Sigma_g$.

\end{packed_enum}
\end{proposition}

\para{Remark} The intersection pairing $H_{n-k} E \otimes H_{k} E \to \Z$ identifies $H_{n-k} E$ with $\Hom(H_k E, \Z)$ and hence with $H^k E$ by the universal coefficients theorem, since the homology of a surface bundle over a surface with monodromy in $\mathcal I_g$ is torsion-free (see Proposition \ref{proposition:morita}). Therefore, Proposition \ref{proposition:cups} can also be viewed as a description of the cup product in $H^*E$. 

\begin{proof}
Before beginning with the proof of the statements, a comment on orientations is in order. Recall that if $X,Y$ are embedded surfaces intersecting transversely, then $X \cap Y$ is oriented via the convention that
\[
N(X) \oplus N(Y) \oplus T (X \cap Y)
\]
should be positively oriented, where, for $W = X$ or $W= Y$, $N(W)$ is oriented by the convention that $N(W) \oplus T(W)$ be positively oriented with respect to the orientation fixed on $W$. Note that relative to this convention, if $X$ is of odd codimension, then $X\cdot X = 0$; we will often employ this fact without comment in the sequel.\\
\indent Recall that the submanifolds $\Sigma_x \subset M_\phi$ and $M_z \subset E$ have been oriented using a ``base-first'' convention; see (\ref{eqn:star}) and (\ref{eqn:starstar}). As remarked already in the proof of Proposition \ref{proposition:degree}, $E$ itself is oriented by selecting orientations for $B$ and $F$. It is a somewhat tedious process to go through and verify the signs on all of the intersections being asserted in this theorem, and we omit the full verification of these results. At the same time, the reader who is interested in verifying the calculations should have no trouble doing so by carefully tracking the orientation conventions we have laid out.\\

\indent It will turn out to be most natural to construct $C$ after verifying the other statements not involving $C$. We begin with computing $\Sigma_{a,z} \cdot \Sigma_{b,w}$. These are represented by surfaces contained in some $E_\alpha, E_\beta$ respectively, where they are tube surfaces constructed from curves $\gamma, \delta$. We can arrange it so that $\alpha, \beta$ intersect transversely, and such that over these points, the surfaces intersect in their tube regions. Following the orientation conventions as above, one verifies that the local intersection at such a point $(p,q)$, written $I_{(p,q)}$ is equal to $-I_pI_q$, where $I_p$ denotes the local intersection of $\alpha, \beta$ relative to the orientation on $B$, and $I_q$ is the local intersection of $\gamma, \delta$ relative to the orientation on $F$. Summing over all local intersections gives the result in the lower-right hand corner of the table in Proposition \ref{proposition:cups}.1.\\
\indent The relation $[F] \cdot \Sigma_{a,z}= 0$ is easy to verify, by taking $[F]$ to be represented by a fiber not contained in the $E_\alpha$ containing $\Sigma_{a,z}$. This same idea also shows $[F]^2 = 0$, by picking representative fibers over distinct points.\\

\indent Let us turn now to Proposition \ref{proposition:cups}.2. If $E_\alpha, E_\beta$ intersect transversely at a point, then $E_\alpha \cap E_\beta =  F$, the fiber over the point of intersection; a check of the orientation conventions shows that the orientation on $F$ given by the intersection convention agrees with the predetermined orientation, so that 
\[
E_a \cdot E_b = i_B(a,b)[F]
\]
as asserted.\\
\indent The manifolds $M_\gamma$ were constructed so as to intersect each $E_b$ in a tube surface, and so the relation
\[
M_z \cdot E_b = \Sigma_{b,z}
\]
can be taken as a definition of the orientation on $\Sigma_{b,z}$. We choose this over the alternative because it can be verified that under this convention, the orientation on $\Sigma_{b,z}$ agrees with the ``base first'' convention discussed above.\\
\indent Now let $M_x, M_y$ be given, and consider $M_x \cdot M_y \cdot [F]$. By perturbing the one-skeleton of $B$, it can be arranged so that the plugs for $M_x$ and $M_y$ are disjoint and so that the cap regions intersect transversely, and so that the representative fiber intersects $M_x, M_y$ in their tube regions. The local picture therefore becomes the intersection of $x$ and $y$ on $F$. A check of the orientation convention then shows
\[
M_x \cdot M_y \cdot [F] = i_F(x,y).
\]
\indent Turning to Proposition \ref{proposition:cups}.3, consider now a four-fold intersection
\[
M_x \cdot M_y \cdot M_z \cdot [E_\beta].
\]
We will assume without further comment that the intersection of representative submanifolds has been made suitably transverse by choosing one-skeleta wisely. The $M_w$ were constructed so that the problem of computing $M_x \cdot M_y \cdot M_z \cdot [E_\beta]$ is exactly the same as the problem of computing the corresponding $\Sigma_x \cdot \Sigma_y \cdot \Sigma_z$ inside the 3-manifold $E_\beta$, up to a sign which records whether the orientation on $M_x \cdot [E_\beta]$ agrees with the orientation on the corresponding $\Sigma_x \subset E_\beta$; the convention $M_x \cdot E_b = \Sigma_{x,b}$ makes this sign positive. Lemma \ref{proposition:derp} shows that within $E_b$, there exist choices of homology classes $\Sigma_x$ such that
\[
\Sigma_x \cdot \Sigma_y \cdot \Sigma_z = \tilde \tau(b)(x \wedge y \wedge z).
\]
Recall from Lemma \ref{proposition:derp} that the $\Sigma_x$'s are obtained by starting with an arbitrary family $\Sigma'_x$, and adding appropriate multiples of $[F]$. By the preceding, if $a \in B$ satisfies $i_B(a,b) = 1$, then
\[
(M_z + E_a) \cdot E_b = M_z \cdot E_b + [F].
\]
This shows that by adding appropriate multiples of $E_{a}$ to $M_z$ (as specified by the formulas in Lemma\ref{proposition:derp}), for a given $b$, the formula
\begin{equation}\label{e1}
M_x \cdot M_y \cdot M_z \cdot [E_\beta] = \tilde\tau(b)(x \wedge y \wedge z)
\end{equation}
can be made to hold. By choosing a symplectic basis for $H_1 B$, this can be made to hold for all $b \in H_1 B$ simultaneously.\\

\indent It therefore remains to construct the class $C$. If $x,y \in H_1\Sigma_g$ satisfy $i_F(x,y) = 1$, then $[F] \cdot M_x \cdot M_y = 1$. Similarly, if $\alpha, \beta$ are loops in $B$ intersecting transversely exactly once, and $M_x$, $M_y$ are as above, then 
\begin{equation}\label{e2}
\Sigma_{\alpha, x} \cdot \Sigma_{\beta, y} = \Sigma_{\alpha, x} \cdot M_x \cdot E_\beta = \pm 1.
\end{equation}
As the space spanned by $[F]$ and the $\Sigma_{b,x}$ classes has codimension one in $H_2 E$, (\ref{e1}) and (\ref{e2}) together show that the space of classes in $H_2 E$ pairing trivially with the space of $M_x$ has dimension at most one. We claim that
\[
C = M_{x_1} \cdot M_{y_1} + \sum_{(b,z) \in \mathcal B \times \mathcal F}\tilde \tau(b)(x_1 \wedge y_1\wedge z) \Sigma_{\hat b \hat z}
\]
has all the required properties; here $\mathcal B, \mathcal F$ are symplectic bases for $H_1 B, H_1 F$, respectively, the map $x \mapsto \hat x$ satisfies $i(x, \hat x) = 1$, $x_1 \in \mathcal B$, and $\hat x_1 = y_1$. Recall that $C$ is asserted to have the following properties: $C \cdot [F] = 1$ and $C \cdot \Sigma_{b,z} = 0$ for all $b \in H_1B, z \in H_1 \Sigma_g$. Additionally, when the monodromy of $E$ is contained in the Johnson kernel, we require $C^2 = 0$ and $C \cdot M_x = 0$ for $M_x$ in the family associated to the lift of $\tau$ to the zero homomorphism. The proof is a direct calculation. For $C\cdot[F]$, one has by Proposition \ref{proposition:cups}.1 and then Proposition \ref{proposition:cups}.2
\begin{align*}
C \cdot [F]  &= \left( M_{x_1} \cdot M_{y_1} + \sum_{(b,z) \in \mathcal B \times \mathcal F} \tilde\tau(b)(x_1 \wedge y_1 z) \Sigma_{\hat b \hat z}\right) \cdot [F]\\
		& = M_{x_1} \cdot M_{y_1} \cdot [F]\\
		& = 1.
\end{align*}
Computation of $C \cdot \Sigma_{b,z}$ proceeds by Proposition \ref{proposition:cups}.3 and Proposition \ref{proposition:cups}.1 respectively.
\begin{align*}
C \cdot \Sigma_{b,z} &= M_{x_1} \cdot M_{y_1} \cdot \Sigma_{b,z} +\tilde\tau(b)(x_1 \wedge y_1 \wedge z) (\Sigma_{\hat b \hat z}) \cdot \Sigma_{b,z}\\
&= \tilde \tau(b)(x_1 \wedge y_1 \wedge z) - \tilde\tau(b)(x_1 \wedge y_1 \wedge z)\\
&= 0.
\end{align*}

\indent When the monodromy of $E$ is contained in $\K_g$, the above formula for $C$ simplifies to $C = M_{x_1} \cdot M_{y_1}$, from which it is apparent that $C^2 = 0$. To see that $C \cdot M_x = 0$ for all $x$, we will apply Poincar\'e duality to see that it suffices to show that
\[
C \cdot M_x \cdot Y = 0
\]
for all classes $Y \in H_3 E$. Since $M_x \cdot E_b = \Sigma_{b x}$ and we have shown $C \cdot \Sigma_{b x} = 0$, it remains only to consider $C \cdot M_z \cdot M_w$. Expanding $M_z \cdot M_w$ in the additive basis for $H_2 E$,
\[
M_z \cdot M_w = \alpha [F] + \beta C + \sum_{(b,z) \in \mathcal B \times \mathcal F} \gamma_{b,z} \Sigma_{\hat b,\hat z}.
\]
As the monodromy of $E$ is contained in $\K_g$, we have $M_z \cdot M_w \cdot \Sigma_{b,x} = 0$; applying this in coordinates for some $(b,x) \in \mathcal B \times \mathcal F$ gives, by applying the prior formulas,
\begin{align*}
0 &= \left(\alpha [F] + \beta C + \sum_{(b,z) \in \mathcal B \times \mathcal F} \gamma_{b,z} \Sigma_{\hat b,\hat z}\right) \cdot \Sigma_{b,x}\\
&= -\gamma_{b,x}, 
\end{align*}
so that all $\gamma_{b,z} = 0$. Consequently, $M_z \cdot M_w = \alpha[F] + \beta C$. Recalling that $[F]^2 = C^2 = 0$ and that $(M_z \cdot M_w)^2 = 0$, this implies $\alpha \beta = 0$. 

Also,
\[
i_F(z,w) = M_z \cdot M_w \cdot [F] = \beta.
\]
Therefore, we conclude that in the case $i_F(z,w) \ne 0$,
\[
M_z \cdot M_w = i_F(z,w) C.
\]
As $C^2 = 0$ this shows the result in this case. Now suppose that $i_F(z,w) = 0$. Then we can find $z'$ such that $M_z \cdot M_{z'} = cC$ by above, with $c \ne 0$, and then
\[
0 = M_z \cdot M_w \cdot M_z \cdot M_{z'} = c M_z \cdot M_w \cdot C.
\]
This shows that $M_z \cdot M_w \cdot C = 0$ for all $z,w$, finishing the proof of Proposition \ref{proposition:cups}.
\end{proof}


\section{Multisections and splittings on rational cohomology}\label{section:appendix}
Let $p: E \to B$ be a surface bundle over an arbitrary base space $B$ equipped with a section $\sigma: B \to E$. Then there is an associated splitting of $H^1(E, \Z)$ as a direct sum, via
\begin{equation} \label{eqn:spl}
H^1(E, \Z) = \im p^* \oplus \ker \sigma^*. 
\end{equation}
The condition that $p: E \to B$ admit a section is restrictive. However, recent work of Hamenst\"adt shows that all surface bundles over surfaces with zero signature admit {\em multisections} (see Theorem \ref{theorem:hamen}). In this section, we develop some necessary machinery showing how a multisection of a surface bundle gives rise to a splitting of $H^1(E, \Q)$, similarly to (\ref{eqn:spl}). The results of this section will be required in the proof of Theorem \ref{theorem:double}.

\begin{remark} Theorem \ref{theorem:hamen} is the only result in this section that requires the base space $B$ to be a surface of genus $g \ge 2$. Lemma \ref{lemma:gamma} and Proposition \ref{proposition:qsection} are valid for any base space $B$.
\end{remark}

 Let $\Conf_n(E)$ denote the configuration space of $n$ unordered distinct points in $E$, and let $\PConf_n(E)$ denote the space of $n$ ordered distinct points in $E$. The symmetric group on $n$ letters $S_n$ acts freely on $\PConf_n(E)$ by permuting the order of the points, and $\PConf_n(E) / S_n = \Conf_n(E)$.

 By a {\em multisection} of $p: E \to B$, we mean a map
\[
\sigma: B \to \Conf_n(E)
\]
for some $n \ge 1$, such that the composition
\[
B \to \Conf_n(E) \to B^n/S_n
\]
is given by $x \mapsto [x,\dots, x]$. In other words, a multisection selects $n$ distinct unordered points in each fiber. A {\em pure multisection} is a map
\[
\sigma: B \to \PConf_n(E)
\]
such that the composition
\[
B \to \PConf_n(E) \to B^n
\]
is given by $x \mapsto (x,\dots,x)$. Our interest in multisections is due to the following result of Hamenst\"adt (see \cite{hamenstadt}):
\begin{theorem}{\bf{(Hamenst\"adt)}}\label{theorem:hamen}
Let $p: E \to B$ be a surface bundle over a surface such that the signature of $E$ is zero (e.g. a bundle with at least one fibering with monodromy lying in $\mathcal I_g$). Then $p: E \to B$ has a multisection $\sigma$ of cardinality $2g -2$. 
\end{theorem}

We will use this result to obtain a splitting on $H^*(E,\Q)$. As (\ref{eqn:spl}) indicates, this is straightforward when the multisection is pure; the work will be to obtain the required maps for general multisections. First note that by taking a finite cover $\tilde B \to B$, we can pull the bundle back to $\tilde p: \tilde E \to \tilde B$, such that the multisection pulls back to a pure multisection:
\[
\psi: \tilde B \to \PConf_n(\tilde E).
\]
Moreover, we can assume that the covering $\tilde B \to B$ is normal, with deck group $\Gamma$. By pulling back the $\Gamma$ action on $\tilde B$, we see that $\Gamma$ also acts on $\tilde E$, by sending the fiber over $b$ to the fiber over $\gamma(b)$. Then the multisection $\psi$ is in fact $\Gamma$-equivariant. This suggests the following lemma.
\begin{lemma}\label{lemma:gamma}
Let $\tilde \sigma: \tilde B \to \tilde E$ be a $\Gamma$-equivariant section. Then there is an induced map on $\Gamma$-invariant cohomology:
\[
\tilde \sigma^*: H^*(\tilde E, \Q)^\Gamma \to H^*(\tilde B, \Q)^\Gamma.
\]
As a result, the transfer map
\[
\tau^*: H^*(\tilde B, \Q) \to H^*(B, \Q)
\]
is injective when restricted to $\tilde \sigma^*(H^*(\tilde E, \Q)^\Gamma)$. 
\end{lemma}

\begin{proof}
If $f: X \to Y$ is any $\Gamma$-equivariant map of topological spaces, then $f^*: H^*Y \to H^*X$ will be equivariant, and so will restrict to a map on the $\Gamma$-invariant subspaces. Transfer (see \cite{hatcher}) gives an identification $H^*(\tilde B, \Q)^\Gamma \approx H^*(B, \Q)$, and the remaining statement follows. 
\end{proof}

We now come to the main result of the section. This asserts that when $p: E \to B$ is a surface bundle with a multisection $\sigma: B \to \Conf_n(E)$, there exists a map $\hat \sigma^*: H^*(B, \Q) \to H^*(E, \Q)$ with many of the same properties as (the pullback of) an actual section map.
\begin{proposition}\label{proposition:qsection}
Suppose $\sigma: B \to \Conf_n(E)$ is a multisection. Then there exist maps
\begin{align*}
\hat \sigma^*:& H^*(E, \Q) \to H^*(B,\Q)\\
\hat \sigma_*:& H_*(B, \Q) \to H_*(E, \Q)
\end{align*}
with the following properties:
\begin{packed_enum}
\item The composition 
\[
\hat \sigma^* \circ p^*: H^*B \to H^* B = \id
\]
and similarly
\[
p_* \circ \hat \sigma_*: H_*B \to H_* B = \id.
\]

\item	The maps $\hat \sigma^*$ and $\hat \sigma_*$ are adjoint under the evaluation pairing. That is, for all $\alpha \in H^*E, x \in H_*B$,
\[
\pair{\alpha, \hat \sigma_* x} = \pair{\hat \sigma^* \alpha, x}.
\]

\item If $\alpha \in \ker \hat \sigma^*$, then for any $\beta \in H^*(E,\Q)$ and any $x \in H_*(B, \Q)$,
\[
\pair{\alpha \smile \beta, \hat \sigma_*(x)} = 0.
\]

\end{packed_enum}
Consequently, $\hat \sigma^*$ induces a splitting
\begin{equation}\label{spl}
H^1(E, \Q) = \im p^* \oplus \ker \hat \sigma^*.
\end{equation}

\end{proposition}

\begin{proof}
Begin by assuming that the multisection is pure. For $i = 1, \dots, n$ let $p_i: \PConf_n(E) \to E$ be the projection onto the $i^{th}$ coordinate. We define
\begin{align*}
\hat \sigma^*(\alpha) &= \frac{1}{n} \sum_{i = 1}^n \sigma^*(p_i^*(\alpha))\\
\hat \sigma_*(x) &= \frac{1}{n} \sum_{i = 1}^n (p_i)_*(\sigma_*(x)).
\end{align*}
Then properties (\ref{proposition:qsection}.1) - (\ref{proposition:qsection}.3) follow by direct verification.

In the general case, let $c: \tilde B \to B$ be a normal covering such that $\sigma$ pulls back to a pure multisection $\psi$. We will use $\bar c$ to denote the covering $\tilde E \to E$. Let $\tau^*: H^*(\tilde B, \Q) \to H^*(B,\Q)$ be the transfer map, normalized so that $c^* \circ \tau^* = \id$. Then define $\hat \sigma^*: H^*(E, \Q) \to H^*(B,\Q)$ by
\[
\hat \sigma^* = \tau^* \circ \hat \psi^* \circ \bar c^*.
\]
Similarly, define $\hat \sigma_*: H_*(B, \Q) \to H_*(E, \Q)$ by
\[
\hat \sigma_* = \bar c_* \circ \hat \psi_* \circ \tau_*.
\]

For what follows, it will be useful to refer to the following diagram.
\[
\xymatrix{
H^*(\tilde E) \ar@<1ex>[d]^{\hat \psi^*} \ar@<1ex>[r]^{\tau^*}	& H^*(E) \ar@<1ex>[l]^{\bar c^*} \ar@<1ex>@{.>}[d]^{\hat \sigma} \\
H^*(\tilde B) \ar@<1ex>[u]^{\tilde p^*} \ar@<1ex>[r]^{\tau^*}	& H^*(B) \ar@<1ex>[l]^{c^*} \ar@<1ex>[u]^{p^*}
}
\]
By definition, 
\[
\hat \sigma^* \circ p^* =  \tau^* \circ \hat \psi^* \circ \bar c^* \circ p^*.
\]
By commutativity, $\bar c^* \circ p^* = \tilde p^* \circ c^*$. Then
\begin{align*}
\tau^* \circ \hat \psi^* \circ \bar c^* \circ p^* &= \tau^* \circ \hat \psi^* \circ \tilde p^* \circ c^*\\
&= \tau^* \circ c^*\\
&= \id.
\end{align*}
Here, we have used the property $\hat \psi^* \circ \tilde p^* = \id$ for the pure multisection $\psi$,  as well as our normalization convention $\tau^* \circ c^* = \id$ for the transfer map. A similar calculation proves the corresponding result for $\hat \psi_*$, and (\ref{proposition:qsection}.1) follows. 

Statement \ref{proposition:qsection}.2 follows from the observation that the cohomology and homology transfer maps are adjoint under the evaluation pairing. That is, if $\tilde X \to X$ is a normal covering space with deck group $\Gamma$, then for $x \in H_*X$ and $\alpha \in H^*\tilde X$,
\[
\pair{\alpha, \tau_* (x)} = \pair{\tau^*(\alpha), x}.
\]
As $\hat \psi^*$ and $\bar c^*$ certainly also enjoy this adjointness property, so does $\hat \sigma^*$, and (\ref{proposition:qsection}.2) follows. 

To establish (\ref{proposition:qsection}.3), suppose $\alpha \in \ker \hat \sigma^*$, and take $\beta \in H^*(E, \Q), x \in H_*(B, \Q)$. As the transfer map is not a ring homomorphism, (\ref{proposition:qsection}.3) does not follow immediately from (\ref{proposition:qsection}.2). However, we see that
\begin{align*}
\pair{\alpha \smile \beta, \hat\sigma_*(x)} &= \pair{\hat \sigma^*(\alpha \smile \beta), x}\\
&= \pair{\tau^*((\hat \psi^* \circ \bar c^*)(\alpha) \smile (\hat \psi^* \circ \bar c^*)(\beta)), x}.
\end{align*}
It therefore suffices to show that $\hat \psi^* \circ \bar c^* (\alpha) = 0$. This follows by Lemma \ref{lemma:gamma}. Indeed, $\bar c ^*(\alpha) \in H^*(\tilde E, \Q)^\Gamma$, and $\hat \psi^*$, being a sum of $\Gamma$-equivariant maps, is itself $\Gamma$-equivariant, and so $\hat \psi^* \circ \bar c^*$ takes image in $H^*(\tilde B, \Q)^\Gamma$. On the one hand, we have
\[
0 = \hat \sigma^*\alpha =  \tau^* \circ \hat \psi^* \circ \bar c^*(\alpha)
\]
by assumption. By Lemma \ref{lemma:gamma}, $\tau^*$ is injective on the image of $\hat \psi^* \circ \bar c^*$, so that $\hat \psi^* \circ \bar c^*(\alpha) = 0$ as desired. \end{proof}


\section{Unique fibering in the Johnson kernel}\label{section:proof}
This section is devoted to the proof of Theorem \ref{theorem:double}. The outline is as follows. Let $p_1: E \to B_1$ be a surface bundle with monodromy in the Torelli group $\mathcal I_g$, and suppose there is a second distinct fibering $p_2: E \to B_2$ with fiber $F_2$. The proof proceeds by analyzing $[F_2]$ in the coordinates on $H_*E$ coming from the Torelli fibering $p_1$. On the one hand, the intersection form in these coordinates is completely understood by virtue of Proposition \ref{proposition:cups}. On the other, $[F_2]$ is realizable as an intersection of classes induced from $H_1 B_2$. Under the assumption that the monodromy of $p_1$ is contained in $\K_g$ and not merely $\mathcal I_g$, it will follow that there is a unique possibility for $[F_2]$. The final step will be to extract the condition that the genera of $F_2$ and $B_1$ must be equal from the cohomology ring $H^*E$ and to argue that this enforces the triviality of either bundle structure.\\

\para{The fundamental class of a second fiber} In this subsection we will compute $[F_2]$ in the coordinates on $H_2$ coming from the fibering $p_1$. The results are formulated under the more general assumption that the monodromy of $p_1$ lie in $\mathcal I_g$ rather than $\mathcal K_g$, because we feel that the arguments are clearer in this larger context. The main objective is Lemma \ref{proposition:ef2}.\\
\indent Suppose that $p_1: E \to B_1$ is a bundle with monodromy lying in $\mathcal I_g$. Choose a lift $\tilde \tau$ of the Johnson homomorphism to $\wedge^3 H$; then by Proposition \ref{proposition:cups}.3, there is a natural splitting
\[
H_3 E \approx p_1^! H_1 B_1 \oplus H_1 F_1
\]
We use this direct sum decomposition to define the projections
\[
P: H_3 E \to p_1^! H_1 B_1 \qquad \mbox{and} \qquad Q: H_3 E \to H_1 F,
\]
and we consider the restrictions of $P$ and $Q$ to $p_2^! H_1 B_2$ for a second fibering $p_2: E \to B_2$. Where convenient, we will also define $P$ and $Q$ on $H_1B_2$ directly, by precomposing with the injection $p^!$. 
\begin{lemma}\label{proposition:Qsym}
For any second fibering $p_2: E \to B_2$, the restriction of $Q$ to $H_1 B_2$ is a symplectic mapping, with respect to $d\,i_{F_1}$ on $H_1 F_1$ and $i_{B_2}$ on $H_1 B_2$, where $d = [F_1]\cdot[F_2]$ is the algebraic intersection number of the two fibers. 
\end{lemma}
\begin{proof}
There exist classes $x,y \in H_1 B_2$ such that $x \cdot y = 1 \in H_0 B_2$, so that $[F_2] = p_2^! x \cdot p_2^! y$, and there are expressions
\[
p_2^! x = Px + Qx, \qquad p_2^! y = Py + Qy.
\]
Consequently,
\[
[F_2] = Px \cdot Py + Px \cdot Qy - Py \cdot Qx + Qx \cdot Qy.
\]
By Proposition \ref{proposition:cups}, $[F_1] \cdot Pz = 0$ for all $z \in H_1 B_2$, so that
\[
d = [F_1] \cdot [F_2] = [F_1] \cdot Qx \cdot Qy,
\]
with the first equality holding by assumption. The condition $[F_2] = p_2^! x \cdot p_2^! y$ is equivalent to $i_{B_2}(x,y) = 1$. By Proposition \ref{proposition:cups},
\[
 d = [F_1] \cdot Qx \cdot Qy = i_{F_1}(Qx, Qy),
\]
proving the claim.
\end{proof}

\indent As in the above proof, let $x,y \in H_1 B_2$ satisfy $x \cdot y = 1$. By Poincar\'e duality, in order to determine $[F_2]$ it suffices to determine the collection of cup products $[F_2]\cdot Z$ for $Z \in H_2 E$. Relative to the splitting of $H_2 E$ coming from $p_1$ (where the monodromy lies in $\mathcal I_g$), in particular we must determine $[F_2]\cdot \Sigma_{b,z}$, where $b \in H_1 B_1$ and $z \in H_1 F_1$. 

\begin{lemma}\label{proposition:etaformula}
Take $x,y \in H_1 B_2$ satisfying $x \cdot y =1$. For $b \in H_1 B_1$ and $z \in H_1 F_1$, let $\Sigma_{b,z}$ be the associated element of $H_2 E$. Then
\begin{equation}\label{equation:etaformula}
[F_2] \cdot \Sigma_{b,z} = i_{B_1}(Px,b)i_{F_1}(Qy,z) - i_{B_1}(Py,b)i_{F_1}(Qx,z) + \tau(b)(Qx \wedge Qy \wedge z).
\end{equation}
In particular, if $z \in \langle Qx, Qy \rangle^\perp$, then (\ref{equation:etaformula}) simplifies to
\begin{equation}\label{equation:etatau}
[F_2] \cdot \Sigma_{b,z} = \tau(b)(Qx \wedge Qy \wedge z).
\end{equation}
In fact, for all $z \in H_1 F_1$, there exists a pair $x_z,y_z \in H_1 B_2$ such that $z \in \langle Qx_z, Qy_z \rangle^\perp$ holds, so that for all $b,z$, (\ref{equation:etatau}) is satisfied for this choice of $x_z,y_z$. 
\end{lemma}
\begin{proof}
The formulas in (\ref{equation:etaformula}) and (\ref{equation:etatau}) follow directly from the description of the intersection form given in Proposition \ref{proposition:cups}. The existence of a suitable $x,y$ for a given $z$ is nothing but a matter of symplectic linear algebra. Since we will use some features of the construction later on, we give a detailed explanation. Lemma \ref{proposition:Qsym} shows that $W = \im Q$ is a symplectic subspace of $H_1 F_1$, and so we can take a symplectic complement $W^\perp$. Any $z$ can therefore be written as $w + w'$ with $w \in W$ and $w' \in W^\perp$. If $w = 0$ there is nothing to show. Otherwise, extend $w$ to a symplectic basis for $W$ so that $w = x_1$. As $B_2$ has genus $\ge 2$, this basis includes $x_2, y_2$, and as $W = \im Q$, we can select $x_z,y_z$ in $H_1 B_2$ with $Qx_z = x_2$ and $Qy_z = y_2$.
\end{proof}

\indent We conclude this subsection by amalgamating the work we have done in the previous two propositions in order to give a description of $[F_2]$. 
\begin{lemma}\label{proposition:ef2}
Let $p_2: E \to B_2$ be a second fibering. The choice of lift $\tilde \tau$ of the Johnson homomorphism furnishes $H_2 E$ with the following splitting 
\[
H_2 E = \pair{[F_1]} \oplus (H_1B_1\otimes H_1 F_1) \oplus H_2 B_1,
\]
with $H_1B_1\otimes H_1 F_1$ spanned by the set of $\Sigma_{b,z}$ where $b,z$ range in symplectic bases $\mathcal{B}, \mathcal{F}$ for $H_1 B_1, H_1 F_1$ respectively, and $H_2 B_1$ spanned by $C$ as in Proposition \ref{proposition:cups}. Relative to this splitting of $H_2E$ there is the following expression for $[F_2]$:
\begin{equation}\label{equation:eta2}
[F_2]= (\delta - 2de) [F_1] + d C + \sum_{b \in \mathcal B, z \in \mathcal F} \tilde\tau(b)(Qx_z \wedge Qy_z \wedge z) \Sigma_{\hat b \hat z}.
\end{equation}
Here, $\delta = i_{B_1}(Px,Py) + Qx \cdot Qy \cdot C$ for any choice of $x,y \in H_1 B_2$ satisfying $x \cdot y = 1$, $e = C^2$, and $d = [F_1] \cdot [F_2]$ (the algebraic intersection of the two fibers). Also $\hat x$ denotes the symplectic dual of $x$ relative to the chosen symplectic basis. 
\end{lemma}
\begin{proof}
Suppose $V$ is a free $\Z$-module equipped with a nondegenerate symmetric bilinear pairing $\pair{\cdot, \cdot}$. Suppose moreover that there exists a generating set $\mathcal A = \{a_1, \dots, a_k, b_1, \dots, b_k\}$ with the property that $\pair{a_i, a_j} = \pair{b_i, b_j} = 0$ for all $i,j$, $\pair{a_i, b_j} = 0$ for $i \ne j$, and $\pair{a_i,b_i} = 1$. Then any element $x \in V$ is expressible in the form
\begin{equation}\label{equation:sumform}
x = \sum_{i = 1}^k \pair{x, a_i} b_i + \sum_{i = 1}^k \pair{x, b_i} a_i. 
\end{equation}
We will apply this to $V = H_2 E$ with the intersection pairing; in order to do this we must find a suitable generating set $\mathcal A$. Via Proposition \ref{proposition:cups}, the space $H_1B_1 \otimes H_1 F_1$ is orthogonal under $\cdot$ to $H_2 B_2$ and to $H_2F_1$, and moreover, the collection of $\Sigma_{b,z}$ for $(b,z) \in \mathcal B \times \mathcal F$ is such a generating set on this subspace. We also have $[F_1]\cdot C = 1$, as well as $([F_1])^2 = 0$ and $C^2 = e$. Therefore, we can take 
\[
\mathcal A = \{[F_1], C- e [F_1]\} \cup \{\Sigma_{b,z} \mid (b,z) \in \mathcal B \times \mathcal F\}.
\]
The only intersection that remains to be computed is $[F_2] \cdot C$. As $Px \cdot Py = i_{B_1}(Px,Py) [F_1]$, a direct computation gives
\begin{align*}
[F_2] \cdot C	&= (Px \cdot Py + Px \cdot Qy - Py \cdot Qx + Qx \cdot Qy) \cdot C\\
			&= Px \cdot Py \cdot C + Qx \cdot Qy \cdot C\\
			&= i_{B_1}(Px, Py) + Qx \cdot Qy \cdot C = \delta.
\end{align*}
By assumption, $[F_1] \cdot [F_2] = d$, and Formula (\ref{equation:etatau}) computes $[F_2] \cdot \Sigma_{b,z}$. Therefore we may insert these computations into Formula (\ref{equation:sumform}) to obtain (\ref{equation:eta2}).
\end{proof}

\para{Rigidity in the Johnson kernel} We now assume, as is required for Theorem \ref{theorem:double}, that the monodromy of $p_1$ is contained in $\K_g$. As noted in the previous section, the closed Johnson kernel $\K_{g}$ is defined to be the kernel of $\tau: \mathcal{I}_{g} \to \wedge^3 H/H$; similarly the pointed Johnson kernel $\K_{g,*}$ is the kernel of $\tau: \mathcal I_{g,*} \to \wedge^3 H$. We also noted above that if $\tau \circ \rho: H_1 B \to \wedge^3 H/ H$ is identically zero then there is a {\em canonical} lift
$\tilde \tau: H_1 B \to \wedge^3 H$, namely zero. This furnishes the (co)homology of $E$ with a canonical splitting in which all cup products in (\ref{equation:etatau}) vanish.\\
\indent In order to prove the main result of this section, we will compute $[F_2]$ and see that it is ``as simple as possible'' in the coordinates coming from $p_1$, the fibering with monodromy in $\mathcal K_g$. This will be accomplished via Lemma \ref{proposition:ef2}. Per our choice of lift $\tilde \tau$, the terms expressed via the Johnson homomorphism all vanish, so that
\[
[F_2] = a[F_1] + d C,
\]
for some $a \in \Z$. The coefficient $a$ is determined by $[F_2] \cdot C$, or equivalently $\delta = i_{B_2}(Px, Py)$ (by Proposition \ref{proposition:cups}.3, the term $Qx \cdot Qy \cdot C = 0$). This can be determined from Lemma \ref{proposition:etaformula}.

\begin{lemma}\label{lemma:P}
Let $E$ be a 4-manifold with two fiberings as a surface bundle over a surface: $p_1: E \to B_1$ and $p_2: E \to B_2$. Define the projection $P: H_1 B_2 \to H_1 B_1$. Suppose the monodromy for the bundle structure associated to $p_1$ lies in $\K_g$. Then $P \equiv 0$, and consequently $\delta = 0$.
\end{lemma}
\begin{proof}
Returning to (\ref{equation:etaformula}), in the Johnson kernel setting, both $[F_2] \cdot \Sigma_{b,z}$ and $\tilde \tau(b)(Qx \wedge Qy \wedge z)$ are zero for all $x,y,z$. Taking $z$ to be any element satisfying $i_{F_1}(Qy,z) \ne 0$ and $i_{F_1}(Qx,z) = 0$, (\ref{equation:etaformula}) simplifies to $i_{B_1}(Px, b) = 0$. Since this is true for all $b$, we conclude that $Px = 0$, and since any $x \in H_1 B_2$ has a suitable $y$ so that (\ref{equation:etaformula}) holds, we conclude that $P \equiv 0$ and $\delta = 0$ as claimed.
\end{proof}

With this in hand, we can apply Lemma \ref{proposition:ef2} (recalling from Proposition \ref{proposition:cups}.3 that $e = 0$) to see that $[F_2]$ is as simple as possible:
\begin{equation}\label{eqn:d}
[F_2] = dC.
\end{equation}
As was noted following the statement of Proposition \ref{proposition:morita}, $[F_2]$ must be a primitive class, and so $d = \pm 1$. We conclude that $d = 1$ (as $d \ge 0$ by Proposition \ref{proposition:degree}). We record this fact for later reference:
\begin{lemma}\label{proposition:d1}
Let $p_1: E \to B_1$ be a surface bundle over a surface with monodromy in $\K_g$. Suppose there is a second fibering $p_2: E \to B_2$. Then 
\[
\deg(p_1 \times p_2) = 1.
\]
Proposition \ref{proposition:degree} asserts the equality of $\deg(p_1 \times p_2)$ with $\deg(p_2|_{F_1}: F_1 \to B_2)$ and with $\deg(p_1|_{F_2}: F_2 \to B_1)$. Consequently
\[
\deg(p_2|_{F_1}: F_1 \to B_2) = \deg(p_1|_{F_2}: F_2 \to B_1) = 1.
\]
\end{lemma}

\begin{remark}\label{remark:wrapup}
Observe that Lemma \ref{proposition:d1} supplies a proof of the missing assertion $(\ref{proposition:jfiber}.1) \implies (\ref{proposition:jfiber}.3)$ in Proposition \ref{proposition:jfiber}, namely that if $E$ is a surface bundle over a surface with monodromy in the Johnson kernel, then any second fibering necessarily yields a bi-projection with nonzero degree. Of course, the assertion that any of the conditions $(\ref{proposition:jfiber}.1), (\ref{proposition:jfiber}.2), (\ref{proposition:jfiber}.3)$, are equivalent to the bundle $E$ being a product is the content of Theorem \ref{theorem:double}. 
\end{remark}

\para{Cohomology - splittings coming from sections} In order to complete the proof of Theorem \ref{theorem:double}, we will combine the work we have done above with an analysis of what the (co)homology of $E$ looks like with respect to the coordinates coming from the second fibering (where the monodromy need not be contained in $\mathcal I_g$. The most convenient setting for this portion of the argument is in the {\em cohomology} ring, so we pause briefly to establish some preliminaries. 

Most of what we have established vis a vis the intersection pairing on $H_*(E)$ is directly portable to the setting of the cup product in cohomology. In particular, the maps
\[
p_i^*: H^*B_i \to H^*E
\]
for $i = 1,2$ are injections. We let $\eta_i \in H^2 B_i$ be an integral generator compatible with the chosen orientations; it is easy to see that $p_i^*(\eta_i)$ is Poincar\'e dual to $[F_i]$. Relative to a choice of splitting
\[
H^1 E = p_1^* H^1 B_1 \oplus H^1F_1,
\]
there are the projection maps $P: H^1 B_2 \to H^1 B_1$ and $Q:H^1 B_2 \to H^1 F_1$, and Lemma \ref{lemma:P} carries over to show that $P \equiv 0$. We can also transport our analysis of the intersection form on $H_* E$. In the cohomological setting, we have proved:

\begin{proposition}\label{proposition:cohoproduct}
Let $F_1 \to E \to B_1$ be a surface bundle over a surface with monodromy in the Johnson kernel $\K_g$. Then $E$ is an integral cohomology $B_1 \times F_1$, i.e. there exists a canonical isomorphism
\[
H^*E \approx H^*B_1 \otimes H^* F_1
\]
as graded rings.
\end{proposition}

We now continue with the proof of Theorem \ref{theorem:double}.
\begin{lemma}\label{proposition:annihilexist}
Suppose that the genus of $B_2$ is strictly smaller than that of $F_1$. Then there exist classes $x, y \in H^1 E$ annihilating $p_2^*H^1 B_2$ (that is, $x \smile p_2^*z = y \smile p_2^*z = 0$ for all $z \in H^1 B$), such that $x \smile y = \Phi_1$, where $\Phi_1 \in H^2(F_1)$ is a generator.
\end{lemma}
\begin{proof}
The cohomological formulation of Lemma \ref{lemma:P} shows that 
\[
p_2^* H^1 B_2 \le H^1 F_1.
\]
By (the cohomological reformulation of) Lemma \ref{proposition:Qsym}, $p_2^* H^1 B_2$ is in fact a {\em symplectic} subspace of $H^1 F$, and so there exists a symplectic complement. We can then take the desired $x,y$ to be suitable elements of this complement. \end{proof}

To finish the proof of Theorem \ref{theorem:double}, we will examine where $x,y$ must live, relative to coordinates on $H^* E$ coming from the fibering $p_2$. At this point, the results of Section \ref{section:appendix} come into play. In particular, (\ref{spl}) endows $H^1(E, \Q)$ with a splitting via
\[
H^1(E, \Q) = \im p^* \oplus \ker \hat \sigma^*.
\]
For the remainder of the proof, we will assume that all of our cohomology groups have rational coefficients.

\begin{lemma}\label{proposition:annihillocation}
Let $p: E \to B$ be any surface bundle over a surface with multisection $\sigma$. Suppose that there exists $x \in H^1 E$ annihilating $p^* H^1 B$. Then $x \in \ker \hat \sigma^*$.
\end{lemma}
\begin{proof}
Write
\[
x = v + p^* b,
\]
with $v \in \ker \hat\sigma^*$ and $b \in H^1 B$. If $b \ne 0$, then there exists $c \in H^1 B$ with $b \smile c \ne 0$. On the one hand, $x \smile p^*c = 0$ by assumption. On the other, letting $[B] \in H_2 B$ denote the fundamental class, we have by Proposition \ref{proposition:qsection}
\begin{align*}
\pair{x \smile p^* c, \hat \sigma_*[B]} &= \pair{(v + p^*b)\smile p^*c, \hat \sigma_*[B]}\\ 
							&= \pair{v \smile p^*c, \hat \sigma_*[B]} + \pair{p^*(b \smile c), \hat \sigma_*[B]}\\
							&= 0 + \pair{\hat \sigma^* p^*(b \smile c), [B]}\\
							&= \pair{b \smile c, [B]} \ne 0, 
\end{align*}
since $v \in \ker \hat \sigma^*$. In this case we have reached a contradiction, and so $b = 0$ as desired. 
\end{proof}

\begin{lemma}\label{proposition:gh}
Let $F_1 \to E \to B_1$ be a surface bundle over a surface with monodromy in $\K_g$, and suppose there is a second fibering $p_2: E \to B_2$. Let $g$ denote the genus of $F_1$, and $h$ denote the genus of $B_2$. Then $g = h$.
\end{lemma}
\begin{proof}
We have already established (Lemma \ref{proposition:d1}) that 
\[
\deg(p_2|_{F_1}) = 1.
\]
As $p_2$ has positive degree, we conclude immediately that $g \ge h$. Suppose $g > h$. Then there exist classes $x,y \in H^1 E$ as in the statement of Lemma \ref{proposition:annihilexist}. We will make use of the existence of a multisection $\sigma$ of $p_2: E \to B_2$, so that by Lemma \ref{proposition:annihillocation}, we must have $x,y \in \ker \hat \sigma^*$. So by Proposition \ref{proposition:qsection},
\[
\pair{x \smile y, \hat\sigma_*[B_2]} = 0.
\]
In the notation of Proposition \ref{proposition:cohoproduct}, both $p_2^* H^1 B_2$ and the classes $x,y$ are contained in $H^1F_1$, and as the image of
\[
\smile: \wedge^2 H^1 F_1 \to H^2 F_1
\]
is one-dimensional (since $F_1$ is a surface), we conclude that $x \smile y = p_2^*(\eta_2)$, where $\eta_2 \in H_2 B_2$ is a generator. So then
\[
\pair{x \smile y, \hat\sigma_*[B_2]} = \pair{p_2^*(\eta_2), \hat\sigma^*[B_2]} = \pair{\eta_2, [B_2]} = 1.
\]
This is a contradiction; necessarily $g = h$.
\end{proof}
This shows that $p_2|_{F_1}$ is a map of degree one between surfaces of the same genus, and as is well-known, therefore
\[
(p_2)_*: \pi_1 F_1 \to \pi_1 B_2
\]
must be an isomorphism. 

\para{Finishing Theorem \ref{theorem:double}} At this point, we turn to an analysis of the fundamental group. Via the long exact sequence in homotopy for a fibration, there is an exact sequence
\[
1 \to \pi_1 F_i \to \pi_1 E \to \pi_1 B_i \to 1,
\]
for $i = 1,2$. Consequently, the kernel of 
\[
(p_1\times p_2)_*: \pi_1 E \to \pi_1 B_1 \times \pi_1 B_2
\]
is given by $\pi_1 F_1 \cap \pi_1 F_2$. On the other hand, this is also the kernel of the cross-projection
\[
\pi_1 F_1 \to \pi_1 B_2
\]
which was just shown to be an isomorphism. We conclude that $(p_1\times p_2)_*$ is an isomorphism.\\
\indent The monodromy of the bundle $E$ can be read off from the fundamental group, as the map $\pi_1 B_1 \to \Out(\pi_1 F_1) \approx \Mod(\Sigma_g)$  (the latter isomorphism coming from the theorem of Dehn, Nielsen, and Baer). Since $\pi_1E$ is a product, this map is trivial. The correspondence (\ref{eqn:correspondence}) then shows that $E$, being a surface bundle with trivial monodromy, is diffeomorphic to $B_1 \times  B_2$. This completes the proof of Theorem \ref{theorem:double}. \qed

    	\bibliography{doublefiber}{}

\begin{thebibliography}{CHR98}

\bibitem[Ati69]{atiyah}
M.~F. Atiyah.
\newblock The signature of fibre-bundles.
\newblock In {\em Global {A}nalysis ({P}apers in {H}onor of {K}. {K}odaira)},
  pages 73--84. Univ. Tokyo Press, Tokyo, 1969.

\bibitem[CHR98]{CHR}
A.~Cavicchioli, F.~Hegenbarth, and D.~Repov{\v{s}}.
\newblock On four-manifolds fibering over surfaces.
\newblock {\em Tsukuba J. Math.}, 22(2):333--342, 1998.

\bibitem[FM12]{FM}
B.~Farb and D.~Margalit.
\newblock {\em A primer on mapping class groups}, volume~49 of {\em Princeton
  Mathematical Series}.
\newblock Princeton University Press, Princeton, NJ, 2012.

\bibitem[GP10]{gp}
V.~Guillemin and A.~Pollack.
\newblock {\em Differential topology}.
\newblock AMS Chelsea Publishing, Providence, RI, 2010.
\newblock Reprint of the 1974 original.

\bibitem[Ham13]{hamenstadt}
U.~Hamenst\"adt.
\newblock On surface subgroups of mapping class groups.
\newblock MSRI workshop video,
  https://www.msri.org/workshops/723/schedules/16639, March 2013.

\bibitem[Hat02]{hatcher}
A.~Hatcher.
\newblock {\em Algebraic topology}.
\newblock Cambridge University Press, Cambridge, 2002.

\bibitem[Hil02]{hillman}
J.~A. Hillman.
\newblock {\em Four-manifolds, geometries and knots}, volume~5 of {\em Geometry
  \& Topology Monographs}.
\newblock Geometry \& Topology Publications, Coventry, 2002.

\bibitem[Joh99]{FEA2}
F.~E.~A. Johnson.
\newblock A rigidity theorem for group extensions.
\newblock {\em Arch. Math. (Basel)}, 73(2):81--89, 1999.

\bibitem[Kod67]{kodaira}
K.~Kodaira.
\newblock A certain type of irregular algebraic surfaces.
\newblock {\em J. Analyse Math.}, 19:207--215, 1967.

\bibitem[Mor87]{morita}
S.~Morita.
\newblock Characteristic classes of surface bundles.
\newblock {\em Invent. Math.}, 90(3):551--577, 1987.

\bibitem[Mor01]{moritabook}
S.~Morita.
\newblock {\em Geometry of characteristic classes}, volume 199 of {\em
  Translations of Mathematical Monographs}.
\newblock American Mathematical Society, Providence, RI, 2001.
\newblock Translated from the 1999 Japanese original, Iwanami Series in Modern
  Mathematics.

\bibitem[Riv11]{rivinfiber}
I.~Rivin.
\newblock Rigidity of fibering.
\newblock Pre-print, http://arxiv.org/abs/1106.4595, 2011.

\bibitem[Riv14]{rivin}
I.~Rivin.
\newblock Statistics of $3$-manifolds fibering over the circle.
\newblock Pre-print, http://arxiv.org/abs/1401.5736, 2014.

\bibitem[Sal14]{salterconstruction}
N.~Salter.
\newblock Surface bundles over surfaces with arbitrarily many fiberings.
\newblock Pre-print, http://arxiv.org/abs/1407.2062, 2014.

\end{thebibliography}
	\bibliographystyle{alpha}

\end{document}